\documentclass[11pt,twoside,reqno]{amsart}
\usepackage{import}
\usepackage{anysize}
\usepackage{amsmath}
\usepackage{amsthm}
\usepackage{amsmath,amscd}
\usepackage[utf8]{inputenc}
\usepackage{amssymb}
\usepackage{stmaryrd}
\usepackage{wasysym}
\usepackage{mathrsfs}
\usepackage[pagebackref=true]{hyperref} 
\renewcommand*{\backref}[1]{}
\renewcommand*{\backrefalt}[4]{({\tiny%
   \ifcase #1 Not cited.%
         \or Cited on page~#2.%
         \else Cited on pages #2.%
   \fi%
   })}
\usepackage{graphicx}
\usepackage{subcaption}
\usepackage{cleveref}
\usepackage{enumitem}
\usepackage{relsize} 
\usepackage{dsfont}
\usepackage{soul}
\usepackage{filecontents}
\hypersetup{colorlinks=true,linkcolor=black,citecolor=black}
\usepackage{color}
\usepackage[all]{xy}
\usepackage{float}
\usepackage{bm}
\usepackage{mathtools}
\usepackage{thmtools}
\usepackage{setspace}
\usepackage{comment}
\usepackage{tikz}
\usepackage{ytableau}
\usepackage{mathdots}
\usepackage[myheadings]{fullpage}

\numberwithin{equation}{section}
\newcommand\mtop{1in}
\newcommand\mbottom{1in}
\newcommand\mleft{1in}
\newcommand\mright{1in}
\usepackage[top = \mtop, bottom = \mbottom, left = \mleft, right=\mright]{geometry}

\newtheorem{thm}{Theorem}[section]

\newtheorem{prop}[thm]{Proposition}

\newtheorem{lemma}[thm]{Lemma}
\newtheorem{cor}[thm]{Corollary}

\theoremstyle{definition}
\newtheorem{defi}[thm]{Definition}

\newtheorem{rmk}[thm]{Remark}

\usepackage{scalerel,stackengine}
\stackMath
\newcommand\reallywidehat[1]{%
\savestack{\tmpbox}{\stretchto{%
  \scaleto{%
    \scalerel*[\widthof{\ensuremath{#1}}]{\kern-.6pt\bigwedge\kern-.6pt}%
    {\rule[-\textheight/2]{1ex}{\textheight}}
  }{\textheight}%
}{0.5ex}}%
\stackon[1pt]{#1}{\tmpbox}%
}
\DeclareSymbolFont{bbold}{U}{bbold}{m}{n}
\DeclareSymbolFontAlphabet{\mathbbold}{bbold}

\makeatletter
\def\@tocline#1#2#3#4#5#6#7{\relax
  \ifnum #1>\c@tocdepth 
  \else
    \par \addpenalty\@secpenalty\addvspace{#2}%
    \begingroup \hyphenpenalty\@M
    \@ifempty{#4}{%
      \@tempdima\csname r@tocindent\number#1\endcsname\relax
    }{%
      \@tempdima#4\relax
    }%
    \parindent\z@ \leftskip#3\relax \advance\leftskip\@tempdima\relax
    \rightskip\@pnumwidth plus4em \parfillskip-\@pnumwidth
    #5\leavevmode\hskip-\@tempdima
      \ifcase #1
       \or\or \hskip 1em \or \hskip 2em \else \hskip 3em \fi%
      #6\nobreak\relax
    \hfill\hbox to\@pnumwidth{\@tocpagenum{#7}}\par
    \nobreak
    \endgroup
  \fi}
\makeatother

\makeatletter
\newcommand{\subalign}[1]{%
  \vcenter{%
    \Let@ \restore@math@cr \default@tag
    \baselineskip\fontdimen10 \scriptfont\tw@
    \advance\baselineskip\fontdimen12 \scriptfont\tw@
    \lineskip\thr@@\fontdimen8 \scriptfont\thr@@
    \lineskiplimit\lineskip
    \ialign{\hfil$\m@th\scriptstyle##$&$\m@th\scriptstyle{}##$\hfil\crcr
      #1\crcr
    }%
  }%
}
\makeatother

\DeclarePairedDelimiter{\abs}{\lvert}{\rvert} 
\DeclarePairedDelimiter{\floor}{\lfloor}{\rfloor}

\newcommand{\R}{\mathbb{R}}
\newcommand{\Z}{\mathbb{Z}}

\newcommand{\N}{\mathbb{N}}
\newcommand{\C}{\mathbb{C}}
\newcommand{\F}{\mathbb{F}}
\newcommand{\D}{\mathbb{D}}
\newcommand{\T}{\mathbb{T}}

\newcommand{\mf}{\mathfrak}

\newcommand{\bbone}{\mathbbold{1}}
\newcommand{\la}{\lambda}
\newcommand{\La}{\Lambda}

\newcommand{\eps}{\epsilon}
\renewcommand{\Re}[1]{\text{Re}(#1)}
\renewcommand{\Re}{\operatorname{Re}}
\renewcommand{\Im}{\text{Im}}

\newcommand{\pfrac}[2]{\left(\frac{#1}{#2}\right)}

\newcommand{\dde}[3]{\left. \frac{d #1}{d #2} \right|_{#3}} 

\newcommand{\ot}{\otimes}

\newcommand{\tmu}{\tilde{\mu}}

\newcommand{\lan}{\left\langle}
\newcommand{\ran}{\right\rangle}

\newcommand{\tth}{^{th}}

\newcommand{\tN}{\tilde{n}}

\newcommand{\tLL}{\tilde{L}}
\newcommand{\Y}{\mathbb{Y}}

\newcommand{\tf}{\tilde{f}}

\newcommand{\bx}{\mathbf{x}}
\newcommand{\by}{\mathbf{y}}
\newcommand{\bi}{\mathbf{i}}
\newcommand{\cP}{\mathbb{Y}}

\renewcommand{\vec}[1]{\boldsymbol{#1}}

\newcommand{\respow}{v}
\DeclareMathOperator{\len}{len}

\DeclareMathOperator{\Res}{Res}

\DeclareMathOperator{\Sig}{Sig}

\DeclareMathOperator{\tG}{\tilde{\Gamma}}

\DeclareMathOperator{\Mat}{Mat}

\DeclareMathOperator{\rank}{rank}

\newcommand{\sqbinom}[2]{\begin{bmatrix}#1\\ #2\end{bmatrix}}

\newcommand{\mbz}{\mathbf{z}}

\newcommand{\GL}{\mathrm{GL}}

\newcommand{\cL}{\mathcal{L}}

\usepackage{pgffor}

\title{The rank of a random triangular matrix over $\F_q$}

\author{Roger Van Peski}
\email{\textcolor{blue}{\href{mailto:rv2549@columbia.edu}{rv2549@columbia.edu}}}

\date{\today}

\begin{document}

\begin{abstract}
We consider uniformly random strictly upper-triangular matrices in $\operatorname{Mat}_n(\mathbb{F}_q)$. For such a matrix $A_n$, we show that $n-\operatorname{rank}(A_n) \approx \log_q n$ as $n \to \infty$, and find that the fluctuations around this limit are finite-order and given by explicit $\mathbb{Z}$-valued random variables. More generally, we consider the random partition whose parts are the sizes of the nilpotent Jordan blocks of $A_n$: its $k$ largest parts (rows) were previously shown by Borodin \cite{borodin1995limit,borodin1999lln} to have jointly Gaussian fluctuations as $n \to \infty$, and its columns correspond to differences $\operatorname{rank}(A_n^{i-1}) - \operatorname{rank}(A_n^i)$. We show the fluctuations of the columns converge jointly to a discrete random point configuration $\mathcal{L}_{t,\chi}$ introduced in \cite{van2023local}. The proofs use an explicit integral formula for the probabilities at finite $n$, obtained by de-Poissonizing a corresponding one in \cite{van2023local}, which is amenable to asymptotic analysis.
\end{abstract}

\maketitle

\tableofcontents

\section{Introduction}

The family of groups $G(n,q)$ of $n \times n$ upper-triangular matrices over $\F_q$ with diagonal entries $1$, is a basic yet rich object. Our goal here is to show that elementary asymptotic questions regarding such matrices lead naturally to a probability distribution defined recently in \cite{van2023local}, which was found there as a limit law for $p$-adic random matrix products, but remains itself somewhat mysterious. 

By now many works in probability have studied Markov chains on $G(n,q)$ and their convergence to the stationary distribution, see for instance Stong \cite{stong1995random}, Coppersmith-Pak \cite{coppersmith2000random}, Arias-Castro-Diaconis-Stanley \cite{arias2004super}, Peres-Sly \cite{peres2013mixing}, Ganguly-Martinelli \cite{ganguly2019upper}, Nestoridi \cite{nestoridi2019super}. From the algebraic side, the character theory and conjugacy classes of these groups are `wild type' and intractable, though various results exist; a nice exposition of these issues and of $G(n,q)$ in general is given in Diaconis-Malliaris \cite{diaconis2021complexity}.




However, the Jordan forms are much simpler than conjugacy classes: all Jordan blocks are unipotent, and hence the Jordan form is parametrized by the multiset of their sizes, which is an integer partition of $n$. Similarly, elements of the corresponding Lie algebra $\mf{g}(n,q)$ of $n \times n$ strictly upper-triangular matrices have only nilpotent Jordan blocks, and the same partition parametrization holds; let us write $J(A)$ for the partition of Jordan block sizes, arranged in decreasing order, as in \Cref{fig:large_jordan}. Concretely, the length of the first row of the partition (i.e. largest Jordan block's size) is the lowest $r$ such that $A^r=0$, while first column, the number of Jordan blocks, is simply the corank of the matrix. More generally, the length of the $i\tth$ column is $\dim \ker(A^i)/\ker(A^{i-1}) = \rank(A^{i-1}) - \rank(A^{i})$, so the lengths of the first $k$ columns are equivalent information to the coranks of the first $k$ powers of the matrix. For $G(n,q)$ the same interpretation of the Jordan block sizes holds after replacing $A$ by $A-I$, so in particular the length of the first column is the dimension of the fixed space of $A$.

Kirillov \cite{kirillov1995variations} asked how the Jordan type of a uniformly random element of $\mf{g}(n,q)$ (equivalently, of $G(n,q)$) is distributed, and how this Jordan type partition `looks' as $n \to \infty$. There are several ways to interpret this question, and Borodin \cite{borodin1995limit,borodin1999lln} answered one by finding the limits of the Jordan block sizes, the rows in \Cref{fig:large_jordan}. Specifically, \cite{borodin1995limit} showed that as $n \to \infty$, the size of the $i\tth$ largest Jordan block grows linearly as $(1-q^{-1})q^{1-i}n$, and the fluctuations of the $k$ largest block sizes are jointly Gaussian. One can already see in \Cref{fig:large_jordan} that these linear growth rates appear to be different, and indeed this is the case.

\begin{figure}
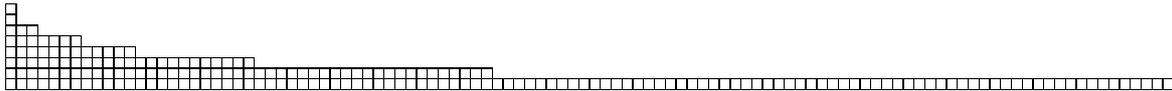


\[
\scalebox{0.24}{
\ydiagram{1,1,3,7,12,23,45,108} 
}
\]

\caption{The Jordan block sizes $J(A) = (108, 45, 23, 12, 7, 3, 1, 1)$ of a uniformly random $A \in \mf{g}(200,2)$, generated on Sage.}\label{fig:large_jordan}

\end{figure}

The present paper answers the dual version of Kirillov's question, determining the asymptotics of the first $k$ columns of the partition in \Cref{fig:large_jordan}, or equivalently the rank of a large triangular matrix and its first $k$ powers. If one samples such a matrix, the Jordan type will look something like \Cref{fig:large_jordan}, with some long rows and much shorter columns, and indeed the asymptotics of the columns differ dramatically from those of the rows. Rather than growing linearly in $n$, they grow as $\log_q n$, and furthermore the order of the fluctuations remains bounded as $n \to \infty$. Hence the limit law of the fluctuations lives not on $\R^k$, but rather on the discrete  set of integer signatures
\begin{equation}
\Sig_k := \{(L_1,\ldots,L_k) \in \Z^k: L_1 \geq \ldots \geq L_k\}.
\end{equation}
In fact, there is not one unique limit law, but rather a family of them, and so we need the following notation.

\begin{defi}\label{def:metric}
For two Borel measures $M_1,M_2$ on a countable set $S$ with discrete $\sigma$-algebra, we denote by 
\begin{equation}
D_\infty(M_1,M_2) := \sup_{x \in S} |M_1(\{x\}) - M_2(\{x\})|
\end{equation}
the $\ell_\infty$ distance between them. When $X_1,X_2$ are two $S$-valued random variables with laws $M_1,M_2$, we abuse notation and write $D_\infty(X_1,X_2) := D_\infty(M_1,M_2)$. Note that the set $S$ is implicit in the notation, and we will often use the same notation $D_\infty$ for different sets $S$.
\end{defi}

The limit object in our main result \Cref{thm:jordan_limit_intro} below is a certain family of distributions $\cL_{q^{-1},\chi} = (\cL^{(i)}_{q^{-1},\chi})_{i \geq 1}$ on infinite integer tuples, which depend on a positive real parameter $\chi$; our reason for delaying its definition is that it requires the notation of Hall-Littlewood symmetric functions. Once this is set up, \Cref{def:cL_series} gives explicit (though quite complicated) formulas for the finite-dimensional distributions of $\cL_{k,q^{-1},\chi} = (\cL^{(i)}_{q^{-1},\chi})_{1 \leq i \leq k}$. 

Though we referred to this family of distributions as the limit object, instead of convergence to a specific member of the family, we have approximation by the appropriate $n$-dependent member of this family with error going to $0$ as $n \to \infty$. Below we use notation $\{x\} := x-\floor{x}$ for the fractional part of a real number $x$. 

\begin{thm}\label{thm:jordan_limit_intro}
Fix a prime power $q$ and for each $n \geq 1$, let $A_n$ be a uniformly random element of the finite set
\begin{equation}
\mf{g}(n,q) := \{A = (a_{i,j})_{1 \leq i,j \leq n} \in \Mat_n(\F_q): a_{i,j} = 0 \text{ for all } i \geq j \}.
\end{equation}
Then for every $k \in \Z_{\geq 1}$,
\begin{equation}\label{eq:haar_metric_cvg}
\lim_{n \to \infty} D_\infty\left((\rank(A_n^{i-1}) - \rank(A_n^i) - \floor{\log_q n})_{1 \leq i \leq k}, (\cL^{(i)}_{q^{-1},q^{\{\log_q n\}}})_{1 \leq i \leq k}\right) = 0,
\end{equation}
where $\cL^{(i)}_{q^{-1},\chi}$ is as defined in \Cref{def:cL_series}, and $D_\infty$ is the metric of \Cref{def:metric} on the set $\Sig_k$.
\end{thm}

In the simplest case $k=1$, which describes corank fluctuations of a large triangular matrix by \Cref{thm:jordan_limit_intro}, the distribution of $\cL^{(1)}_{q^{-1},\chi}$ is fairly explicit:

\begin{equation}\label{eq:cL_1_explicit}
\Pr(\cL^{(1)}_{q^{-1},\chi} = x) = \frac{1}{\prod_{i \geq 1} (1-q^{-i})} \sum_{m \geq 0} e^{-\chi q^{m-x}} \frac{(-1)^m q^{-\binom{m}{2}}}{\prod_{j=1}^m (1-q^{-j})} \quad \quad \quad \quad \text{ for any $x \in \Z$.}
\end{equation}

Further properties of this random variable and its $k > 1$ generalizations are discussed in \cite[Section 1.2]{van2023local}, and simple exact formulas for its exponential moments are given in \cite[Section 7]{van2023local}. One source of our interest in them comes from their appearance in limits of products of random $p$-adic matrices as both the matrix size and number of products are sent to $\infty$ simultaneously, proven in \cite{van2023local}. Though this is not entirely visible from our proofs, in our view the appearance of the same distribution for triangular matrices over a finite field and matrix products over $\Z_p$ takes its origin in two facts. The first is that the Jordan form of $A_n$ parametrizes the $\F_q[T]$-module structure on $\F_q^n$ where $T$ acts by $A_n$, and modules over the ring of integers of a non-archimedean local field have the same structure for such rings as they do for $\Z_p$ and finite extensions thereof \cite[Chapter II]{mac}. The second is that $\cL_{q^{-1},\chi}$ appears---still conjecturally---to be a universal object in this setting, so even though the prelimit probability measures on modules are different in the two settings, this difference between them is irrelevant in the limit. 

\begin{rmk}
It is also natural to study Jordan forms of uniformly random matrices from larger sets such as $\Mat_n(\F_q)$ or $\GL_n(\F_q)$, and many works do so, see Fulman \cite{fulman_main} for a nice survey. In this case, the Jordan block partition associated to a given irreducible polynomial (e.g. the partition of nilpotent Jordan block sizes) does not grow in size as $n \to \infty$, so the asymptotics are quite different from the ones in the present paper or in \cite{borodin1995limit}. 
\end{rmk}

The differences between growth of columns in \Cref{thm:jordan_limit_nometric} and the growth of rows in \cite{borodin1995limit} mean that entirely different methods are needed for the proof. The approach of \cite{borodin1995limit} is to define a growth process on partitions by successively adding independent random columns (and zero rows, to stay square) to the right and bottom of an $n \times n$ upper-triangular matrix to create an $(n+1) \times (n+1)$ matrix, as is done in \cite{borodin1995limit}. For such a sequence of matrices $A_n,n=0,1,2,\ldots$ defined in this way, $J(A_n)$ is a stochastic process on partitions in discrete `time' $n$. One can then find explicit transition probabilities for the steps of this growth process (this was already done in \cite{kirillov1995variations}), show that the rows in \Cref{fig:large_jordan} are not affected by one anothers' positions in the limit because the linear growth rates are different, and then apply standard results on random walks. This approach was used for the rows of similar measures on partitions by Bufetov-Petrov \cite{bufetov2015lln}, F{\'e}ray-M{\'e}liot \cite{feray2012asymptotics}, and the author \cite{van2020limits}. However, as \Cref{thm:jordan_limit_nometric} shows, the $\log_q n$ growth rates of all columns are the same, and indeed the explicit sampling algorithm given in \cite[Theorem 2.3]{borodin1999lln} (stated later as \Cref{thm:borodin_division}) shows that the positions of columns do continue to affect one another as $n \to \infty$. This makes such a hands-on approach untenable, at least as far as we can see.

Instead, we use the fact that this process $J(A_n), n =0,1,2,\ldots$ is a time-discretization of a so-called Hall-Littlewood process in continuous time, see \Cref{sec:macdonald_background} for definitions. This allows one to derive explicit contour integral formulas for its finite-$n$ distributions by `de-Poissonizing' corresponding ones for this Hall-Littlewood process, which we do in \Cref{sec:depoissonize}. We note that for random partitions distributed by the Plancherel measure of the symmetric group, which has similar algebraic structure to our measure but quite different asymptotics, this trick of passing between Poissonized and de-Poissonized ensembles was used in the celebrated works of Baik-Deift-Johansson \cite{baik1999distribution,baik2000distribution} and Borodin-Okounkov-Olshanski \cite{borodin2000asymptotics}.

The prelimit formulas we obtain by de-Poissonization are still quite nontrivial to analyze. Luckily, a similar analysis was recently carried out for integral formulas corresponding to the continuous-time Poissonized version in \cite{van2023local}. In \Cref{sec:asymptotics} we carry out the analogous limit in our setting: many parts of the computation can be (and were) copied from \cite{van2023local}, but our integrand has extra poles which create technical complications. These issues, and the extra steps needed to treat them, are detailed at the beginning of \Cref{sec:asymptotics}.


\textbf{Acknowledgments.} I thank Alexei Borodin for useful discussions and comments on the exposition, and the anonymous referee for helpful comments on the text. This work was supported by the European Research Council (ERC), Grant Agreement No. 101002013.

\section{Symmetric function background}\label{sec:macdonald_background}

This section is mostly a pared-down version of \cite[Section 3]{van2023local}. We denote by $\cP$ the set of all integer partitions $(\la_1,\la_2,\ldots)$, i.e. sequences of nonnegative integers $\la_1 \geq \la_2 \geq \cdots$ which are eventually $0$. We often drop trailing zeroes and write such partitions as $n$-tuples where $n$ is large enough. They may also be represented by Ferrers diagrams as in \Cref{fig:ferrers}.

\begin{figure}[H]
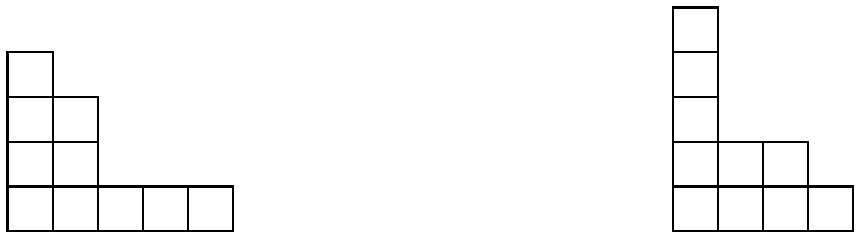

\begin{equation*}
    \raisebox{-.5\height}{\begin{ytableau}
        *(white) \\ 
        *(white) & *(white) \\ 
        *(white) & *(white) \\  
        *(white) & *(white) & *(white) & *(white) & *(white)
    \end{ytableau}}
    \quad \quad \quad \quad \quad \quad \quad \quad \quad \quad \quad \quad \quad \quad \quad 
    \raisebox{.5\height}{\begin{ytableau}
        *(white) \\ 
        *(white) \\  
        *(white) \\ 
        *(white) & *(white) & *(white) \\ 
        *(white) & *(white) & *(white) & *(white)
    \end{ytableau}}
\end{equation*}
\caption{The Young diagram of $\la = (5,2,2,1) $ (left), and that of its conjugate partition $\la' = (4,3,1,1,1)$ obtained by flipping the diagram across the diagonal.} \label{fig:ferrers}
\end{figure}

We call the integers $\la_i$ the \emph{parts} of $\la$, set $\la_i' = \#\{j: \la_j \geq i\}$, and write $m_i(\la) = \#\{j: \la_j = i\} = \la_i'-\la_{i+1}'$. We write $\len(\la)$ for the number of nonzero parts, and denote the set of partitions of length $\leq n$ by $\cP_n$. We write $\mu \prec \la$ or $\la \succ \mu$ if $\la_1 \geq \mu_1 \geq \la_2 \geq \mu_2 \geq \cdots$, and refer to this condition as \emph{interlacing}. Another partial order is defined by containment of Ferrers diagrams, which we write as $\mu \subset \la$, meaning $\mu_i \leq \la_i$ for all $i$. Finally, we denote the partition with all parts equal to zero by $\emptyset$. 

We denote by $\La_n$ the ring $\C[x_1,\ldots,x_n]^{S_n}$ of symmetric polynomials in $n$ variables $x_1,\ldots,x_n$. For a symmetric polynomial $f$, we will often write $f(\bx)$ for $f(x_1,\ldots,x_n)$ when the number of variables is clear from context. We will also use the shorthand $\bx^\la:= x_1^{\la_1} x_2^{\la_2} \cdots x_n^{\la_n}$ for $\la \in \cP_n$. A simple $\C$-basis for $\La_n$ is given by the \emph{monomial symmetric polynomials} $\{m_\la(\bx): \la \in \Y_n\}$ defined by 
\[
m_\la(\bx) = \sum_{\sigma \in S_n/\text{Stab}(\la)} \sigma(\bx^\la)
\]
where $\sigma$ acts by permuting the variables, and $\text{Stab}(\la)$ is the subgroup of permutations such that $\sigma(\bx^\la) = \bx^\la$. It is also a very classical fact that the power sum symmetric polynomials 
\[p_k(\bx) = \sum_{i=1}^n x_i^k, k =1,\ldots,n\]
are algebraically independent and algebraically generate $\La_n$, and so by defining 
\begin{equation*}
    p_\la(\bx) := \prod_{i \geq 1} p_{\la_i}(\bx)
\end{equation*}
for $\la \in \Y$ with $\la_1 \leq n$, we have that $\{p_\la(\bx): \la_1 \leq n\}$ forms another basis for $\La_n$. 

Another special basis for $\La_n$ is given by the \emph{Macdonald polynomials} $P_\la(\bx;q,t)$, which depend on two additional parameters $q$ and $t$ which may in general be complex numbers, though in probabilistic contexts we take $q,t \in (-1,1)$. Our first definition of them requires a certain scalar product on $\La_n$.

\begin{defi}[{\cite[Chapter VI, (9.10)]{mac}}]\label{def:torus_product}
For polynomials $f,g \in \La_n$, define
\begin{equation}
\label{eq:torus_product}
\lan f, g \ran_{q,t;n}' := \frac{1}{n! (2 \pi \bi)^n} \int_{\T^n} f(z_1,\ldots,z_n) \overline{g(z_1,\ldots,z_n)} \prod_{1 \leq i \neq j \leq n} \frac{(z_iz_j^{-1};q)_\infty}{(tz_iz_j^{-1};q)_\infty} \prod_{i=1}^n \frac{dz_i}{z_i},
\end{equation}
where $\T$ denotes the unit circle with usual counterclockwise orientation, and to avoid confusion we clarify that the product is over $\{(i,j) \in \Z: 1 \leq i,j \leq n, i \neq j\}$.
\end{defi}

\begin{defi}\label{def:mac_poly_torus}
The \emph{Macdonald symmetric polynomials} $P_\la(x_1,\ldots,x_n;q,t), \la \in \Y_n$ are defined by the following two properties:
\begin{enumerate}
\item They are `monic' and upper-triangular with respect to the $m_\la(\bx)$ basis, in the sense that they expand as 
\begin{equation}
P_\la(\bx;q,t) = m_\la(\bx) + \sum_{\mu \subsetneq \la} R_{\la \mu}(q,t) m_\mu(\bx).
\end{equation}
\item They are orthogonal with respect to $\lan \cdot, \cdot \ran_{q,t;n}'$. 
\end{enumerate}
\end{defi}

These conditions \emph{a priori} overdetermine the set $\{P_\la(\bx;q,t): \la \in \Y_n\}$, and it is a theorem which follows from \cite[VI (4.7)]{mac} that the Macdonald symmetric polynomials do indeed exist. It is then also clear that they form a basis for $\Y_n$. We note that this paper will only ever consider the two special cases of Macdonald polynomials when either $q$ or $t$ is set to $0$, which are called \emph{Hall-Littlewood polynomials} and \emph{$q$-Whittaker polynomials} respectively.

\begin{defi}\label{def:Q}
For $\la \in \Y_n$, the \emph{dual Macdonald polynomial} $Q_\la$ is given by 
\begin{equation}
Q_\la(x_1,\ldots,x_n;q,t) := b_\la(q,t) P_\la(x_1,\ldots,x_n;q,t)
\end{equation}
where $b_\la$ is an explicit constant given in \cite[p339, (6.19)]{mac}
\end{defi}

The constant multiples $b_\la(q,t)$ are chosen so that the $Q_\la(\bx;q,t)$ form a dual basis to the $P_\la(\bx;q,t)$ with respect to a different scalar product which is related to (a renormalized version of) $\lan \cdot, \cdot \ran_{q,t;n}'$ in the $n \to \infty$ limit, see \cite[Chapter VI, (9.9)]{mac}. Because the $P_\la$ form a basis for the vector space of symmetric polynomials in $n$ variables, there exist symmetric polynomials $P_{\la/\mu}(x_1,\ldots,x_{n-k};q,t) \in \La_{n-k}$ indexed by $\la \in \Y_{n+k}, \mu \in \Y_n$ which are defined by
\begin{equation}\label{eq:def_skewP}
    P_\la(x_1,\ldots,x_{n+k};q,t) = \sum_{\mu \in \Y_n} P_{\la/\mu}(x_{n+1},\ldots,x_{n+k};q,t) P_\mu(x_1,\ldots,x_n;q,t).
\end{equation}
It follows easily from \eqref{eq:def_skewP} that for any $1 \leq d \leq k-1$,
\begin{equation}\label{eq:gen_branch}
    P_{\la/\mu}(x_1,\ldots,x_k;q,t) = \sum_{\nu \in \Y_d} P_{\la/\nu}(x_{d+1},\ldots,x_k;q,t) P_{\nu/\mu}(x_1,\ldots,x_d;q,t).
\end{equation}
We define $Q_{\la/\mu}$ by \eqref{eq:def_skewP} with $Q$ in place of $P$, and it is similarly clear that \eqref{eq:gen_branch} holds for $Q$. An important property of (skew) Macdonald polynomials for probabilistic purposes is the \emph{Cauchy identity} below.

\begin{prop}\label{thm:finite_cauchy}
Let $\nu, \mu \in \Y$. Then
\begin{multline}\label{eq:finite_cauchy}
    \sum_{\kappa \in \Y} P_{\kappa/\nu}(x_1,\ldots,x_n;q,t)Q_{\kappa/\mu}(y_1,\ldots,y_m;q,t) \\
    = \prod_{\substack{1 \leq i \leq n \\ 1 \leq j \leq m}} \frac{(tx_iy_j;q)_\infty }{(x_iy_j;q)_\infty} \sum_{\la \in \Y} Q_{\nu/\la}(y_1,\ldots,y_m;q,t) P_{\mu/\la}(x_1,\ldots,x_n;q,t).
\end{multline}
\end{prop}

The above identity should be interpreted as an identity of formal power series in the variables, after expanding the $1/(1-q^\ell x_iy_j)$ factors as geometric series. It may be seen as partial motivation for the definition of the $Q$ polynomials earlier: the constant factors there are the ones needed for such an identity to hold. For later convenience we set
\begin{equation}\label{eq:def_cauchy_kernel}
     \Pi_{q,t}(\bx;\by) := \prod_{\substack{1 \leq i \leq n \\ 1 \leq j \leq m}} \frac{(tx_iy_j;q)_\infty }{(x_iy_j;q)_\infty} = \exp\left(\sum_{\ell = 1}^\infty \frac{1-t^\ell}{1-q^\ell}\frac{1}{\ell}p_\ell(\bx)p_\ell(\by)\right),
\end{equation}
where the second equality in \eqref{eq:def_cauchy_kernel} is not immediate but is shown in \cite{mac}. 

The skew Macdonald polynomials may also be made explicit, which is needed for later computations.

\begin{defi}\label{def:psi_varphi}
For $\la, \mu \in \Y$, let $f(u) := (tu;q)_\infty/(qu;q)_\infty$ and define 
\begin{equation}\label{eq:pbranch}
    \psi_{\la/\mu}(q,t) :=  \bbone[\mu \prec \la]\prod_{1 \leq i \leq j \leq \len(\la) } \frac{f(t^{j-i}q^{\mu_i-\mu_j})f(t^{j-i}q^{\la_i-\la_{j+1}})}{f(t^{j-i}q^{\la_i-\mu_j})f(t^{j-i}q^{\mu_i-\la_{j+1}})}
\end{equation}
and
\begin{equation}\label{eq:qbranch}
    \varphi_{\la/\mu}(q,t) := \bbone[\mu \prec \la]\prod_{1 \leq i \leq j \leq \len(\mu)} \frac{f(t^{j-i}q^{\la_i-\la_j})f(t^{j-i}q^{\mu_i-\mu_{j+1}})}{f(t^{j-i}q^{\la_i-\mu_j})f(t^{j-i}q^{\mu_i-\la_{j+1}})}.
\end{equation}
When $q$ and $t$ are clear from context we will often write $\psi_{\la/\mu}$ and $\varphi_{\la/\mu}$ without arguments.
\end{defi}

\begin{prop}[{\cite[VI.7, (7.13) and (7.13')]{mac}}] \label{thm:branching_formulas}
For $\la \in \Y_n, \mu \in \Y_{n-k}$, we have
\begin{equation}\label{eq:skewP_branch_formula}
    P_{\la/\mu}(x_1,\ldots,x_k;q,t) = \sum_{\mu = \la^{(0)} \prec \la^{(1)} \prec \cdots \prec \la^{(k)}= \la} \prod_{i=1}^{k-1} x_i^{|\la^{(i+1)}|-|\la^{(i)}|}\psi_{\la^{(i+1)}/\la^{(i)}}
\end{equation}
and
\begin{equation}\label{eq:skewQ_branch_formula}
    Q_{\la/\mu}(x_1,\ldots,x_k;q,t) = \sum_{\mu = \la^{(0)} \prec \la^{(1)} \prec \cdots \prec \la^{(k)}=\la} \prod_{i=1}^{k-1} x_i^{|\la^{(i+1)}|-|\la^{(i)}|}\varphi_{\la^{(i+1)}/\la^{(i)}}.
\end{equation}
\end{prop}

The explicit forms of the Hall-Littlewood and $q$-Whittaker special cases of the formulas in \Cref{def:psi_varphi} will be useful, and are easy to establish by direct computation from \Cref{def:psi_varphi}. 

\begin{lemma}\label{thm:hl_qw_branch_formulas}
Let $\la,\mu \in \Y$ with $\mu \prec \la$. In the Hall-Littlewood case $q=0$ the formulas of \Cref{def:psi_varphi} specialize to
\begin{align}\label{eq:HL_branch}
\begin{split}
\psi_{\la/\mu}(0,t) &= \prod_{\substack{i > 0\\ m_i(\mu) = m_i(\la)+1}} (1-t^{m_i(\mu)})\\ 
\varphi_{\la/\mu}(0,t) &= \prod_{\substack{i > 0\\ m_i(\la) = m_i(\mu)+1}} (1-t^{m_i(\la)}).
\end{split}
\end{align}
In the $q$-Whittaker case $t=0$ they specialize to 
\begin{align}\label{eq:qw_branch}
\begin{split}
\psi_{\la/\mu}(q,0) &= \prod_{i=1}^{\len(\mu)} \sqbinom{\la_i-\la_{i+1}}{\la_i-\mu_i}_q\\ 
\varphi_{\la/\mu}(q,0) &= \frac{1}{(q;q)_{\la_1-\mu_1} }\prod_{i=1}^{\len(\la)-1} \sqbinom{\mu_i-\mu_{i+1}}{\mu_i-\la_{i+1}}_q.
\end{split}
\end{align}
For pairs $\mu,\la$ with $\mu \not \prec \la$, all of the above are $0$.
\end{lemma}

We will require two, in a sense orthogonal, extensions of Macdonald polynomials: symmetric Laurent polynomials in finitely many variables, and symmetric functions---informally, symmetric polynomials in infinitely many variables. 

\subsection{Symmetric Laurent polynomials.} We wish to extend the indices of Macdonald polynomials in $n$ variables from the set $\Y_n$ of partitions of length at most $n$ to the set $\Sig_n$ of signatures of length $n$, where we recall $\Sig_n := \{(x_1,\ldots,x_n) \in \Z^n: x_1 \geq \ldots \geq x_n\}$. We also often identify $\Y_n$ with the subset
\begin{equation}
    \Sig_n^+ := \{(x_1,\ldots,x_n) \in \Z^n: x_1 \geq \ldots \geq x_n \geq 0\} \subset \Sig_n
\end{equation}
of nonnegative signatures. Recall that for signatures $\mu \in \operatorname{Sig}_{k-1}, \lambda \in \operatorname{Sig}_{k}$ we write
$$
\mu \prec \lambda \Longleftrightarrow \lambda_{1} \geq \mu_{1} \geq \lambda_{2} \geq \ldots \geq \mu_{k-1} \geq \lambda_{k},
$$
$|\mu|=\sum_{i} \mu_{i}$, and $\mu-(d[k-1])=\left(\mu_{1}-d, \ldots, \mu_{k-1}-d\right)$.

\begin{lemma}\label{thm:P_shift_property}
Let $\la \in \Y_n, \mu \in \Y_{n-k}$, and $d \in \Z_{\geq 0}$. Then 
\begin{equation}
P_{(\la+(d[n]))/(\mu+(d[n-k]))}(x_1,\ldots,x_k;q,t) = (x_1 \cdots x_k)^d P_{\la/\mu}(x_1,\ldots,x_k;q,t).
\end{equation}
\end{lemma}
\begin{proof}
The claim follows from \eqref{eq:skewP_branch_formula} together with the observation from the explicit formula \eqref{eq:pbranch} that for $\la^{(i)}$ as in \eqref{eq:skewP_branch_formula},
\begin{equation}\label{eq:psi_invariant}
\psi_{\la^{(i+1)}/\la^{(i)}}(q,t) = \psi_{(\la^{(i+1)}+(d[n-k+i+1]))/(\la^{(i)}+(d[n-k+i]))}.
\end{equation}
\end{proof}

This motivates the following.

\begin{defi}\label{def:mac_laurent}
For any $\la \in \Sig_n, \nu \in \Sig_{n-1}$ we define
\begin{equation}\label{eq:pbranch_sig}
    \psi_{\la/\nu}(q,t) :=  \bbone[\nu \prec \la]\prod_{1 \leq i \leq j \leq n-1} \frac{f(t^{j-i}q^{\nu_i-\nu_j})f(t^{j-i}q^{\la_i-\la_{j+1}})}{f(t^{j-i}q^{\la_i-\nu_j})f(t^{j-i}q^{\nu_i-\la_{j+1}})}
\end{equation}
and for $\la \in \Sig_n, \mu \in \Sig_{n-k}$ define 
\begin{equation}\label{eq:skewP_branch_formula_sig}
    P_{\la/\mu}(x_1,\ldots,x_k;q,t) := \sum_{\mu = \la^{(0)} \prec \la^{(1)} \prec \cdots \prec \la^{(k)}= \la} \prod_{i=1}^{k-1} x_i^{|\la^{(i+1)}|-|\la^{(i)}|}\psi_{\la^{(i+1)}/\la^{(i)}}.
\end{equation}
\end{defi}

\begin{cor}\label{thm:signature_shift}
Let $\la \in \Sig_n, \mu \in \Sig_{n-k}$, and $d \in \Z$. Then 
\begin{equation}
P_{(\la+(d[n]))/(\mu+(d[n-k]))}(x_1,\ldots,x_k;q,t) = (x_1 \cdots x_k)^d P_{\la/\mu}(x_1,\ldots,x_k;q,t).
\end{equation}
\end{cor}
\begin{proof}
Same as \Cref{thm:P_shift_property}.
\end{proof}

\begin{rmk}\label{rmk:don't_sig_Q}
We have not stated \Cref{thm:P_shift_property} and \Cref{def:mac_laurent} for the dual Macdonald polynomials, for the simple reasons that (1) the naive versions do not hold, and (2) we do not need this for our proofs. In fact, $\varphi_{\la/\nu}(q,t)$ and the skew polynomials $Q_{\la/\mu}(\bx;q,t)$ must be reinterpreted in a more nontrivial way in order to make such statements true, see \cite[Section 2.1]{van2020limits}.
\end{rmk}

\begin{rmk}
Note that for $\la \in \Sig_n, \mu \in \Sig_{n-1}$ the formula for $\psi_{\la/\mu}(q,0)$ in \Cref{thm:hl_qw_branch_formulas} continues to makes sense, while for the Hall-Littlewood case one should instead interpret
\begin{equation}\label{eq:hl_sig_pbranch}
\psi_{\la/\mu}(0,t) = \prod_{\substack{m_i(\mu) = m_i(\la)+1}} (1-t^{m_i(\mu)})
\end{equation}
(without the $i > 0$ restriction in the product) in order for the translation-invariance property \eqref{eq:psi_invariant} to hold. If $\la \in \Sig_n^+$ (and hence $\mu \in \Sig_{n-1}^+$ by interlacing) then both \eqref{eq:HL_branch} and \eqref{eq:hl_sig_pbranch} give the same result.
\end{rmk}

The defining orthogonality property of Macdonald polynomials also extends readily to their Laurent versions.

\begin{prop}\label{thm:laurent_orthogonality}
If $\la,\mu \in \Sig_n$ and $\la \neq \mu$, then 
\begin{equation}\label{eq:laurent_orthogonality_P}
\lan P_\la(\mbz;q,t), P_\mu(\mathbf{z};q,t) \ran'_{q,t;n} = 0.
\end{equation}
\end{prop}
\begin{proof}
Let $D \in \Z$ be such that $\la + (D[n])$ and $\mu+(D[n])$ both lie in $\Sig_n^+$. Then 
\begin{equation}
\lan P_{\la + (D[n])}(\mbz;q,t), P_{\mu+(D[n])}(\mbz;q,t) \ran'_{q,t;n} = 0
\end{equation} 
by the defining orthogonality property of Macdonald polynomials. However,
\begin{align}
\begin{split}
P_{\la + (D[n])}(\mbz;q,t) \overline{P_{\mu + (D[n])}(\mbz)} &= (z_1 \cdots z_n)^D P_\la(\mbz;q,t) \overline{(z_1 \cdots z_n)^D P_\mu(\mbz;q,t)} \\ 
&= P_\la(\mbz;q,t)\overline{P_\mu(\mbz;q,t)}
\end{split}
\end{align}
for any $z_1,\ldots,z_n \in \T$, so 
\begin{equation}
\lan P_{\la + (D[n])}(\mbz;q,t), P_{\mu+(D[n])}(\mbz;q,t) \ran'_{q,t;n} = \lan P_\la(\mbz;q,t), P_\mu(\mbz;q,t) \ran'_{q,t;n},
\end{equation} 
which completes the proof.
\end{proof}

\subsection{Symmetric functions.} It is often convenient to consider symmetric polynomials in an arbitrarily large or infinite number of variables, which we formalize as follows, heavily borrowing from the introductory material in \cite{van2022q}. One has a chain of maps
\[
\cdots \to \La_{n+1} \to \La_n \to \La_{n-1} \to \cdots \to 0
\]
where the map $\La_{n+1} \to \La_n$ is given by setting $x_{n+1}$ to $0$.
In fact, writing $\La_n^{(d)}$ for symmetric polynomials in $n$ variables of total degree $d$, one has 
\[
\cdots \to \La_{n+1}^{(d)} \to \La_n^{(d)} \to \La_{n-1}^{(d)} \to \cdots \to 0
\]
with the same maps. The inverse limit $\La^{(d)}$ of these systems may be viewed as symmetric polynomials of degree $d$ in infinitely many variables. From the ring structure on each $\La_n$ one gets a natural ring structure on $\La := \bigoplus_{d \geq 0} \La^{(d)}$, and we call this the \emph{ring of symmetric functions}. Because $p_k(x_1,\ldots,x_{n+1}) \mapsto p_k(x_1,\ldots,x_n)$ and $m_\la(x_1,\ldots,x_{n+1}) \mapsto m_\la(x_1,\ldots,x_n)$ (for $n \geq \len(\la)$) under the natural map $\La_{n+1} \to \La_n$, these families of symmetric polynomials define symmetric functions $p_k, m_\la \in \La$. An equivalent definition of $\La$ is $\Lambda := \C[p_1,p_2,\ldots]$ where $p_i$ are indeterminates; under the natural map $\Lambda \to \Lambda_n$ one has $p_i \mapsto p_i(x_1,\ldots,x_n)$.

The Macdonald polynomials satisfy a consistency property 
\begin{equation}
P_\la(x_1,\ldots,x_n,0;q,t) = P_\la(x_1,\ldots,x_n;q,t)
\end{equation}
for any $\la \in \Y$ (and similarly for the dual and skew polynomials). Hence here exist \emph{Macdonald symmetric functions}, denoted $P_\la,Q_\la$ as well, such that $P_\la \mapsto P_\la(\bx;q,t)$ under the natural map $\Lambda \to \Lambda_n$. Macdonald symmetric functions satisfy the skew Cauchy identity
\begin{multline}\label{eq:infinite_cauchy}
    \sum_{\kappa \in \Y} P_{\kappa/\nu}(\bx;q,t)Q_{\kappa/\mu}(\by;q,t) \\
    = \exp\left(\sum_{\ell = 1}^\infty \frac{1-t^\ell}{1-q^\ell} \frac{1}{\ell}p_\ell(\bx)p_\ell(\by)\right) \sum_{\la \in \Y} Q_{\nu/\la}(\by;q,t) P_{\mu/\la}(\bx;q,t).
\end{multline}
Here $P_{\kappa/\nu}(\bx;q,t)$ is an element of $\La$, a polynomial in $p_1(\bx),p_2(\bx),\ldots \in \La$, and summands such as $P_{\kappa/\nu}(\bx;q,t)Q_{\kappa/\mu}(\by;q,t)$ are interpreted as elements of a ring $\La \ot \La$ and both sides interpreted as elements of a completion thereof.

To get a probability measure on $\cP$ from the skew Cauchy identity, we would like homomorphisms $\phi: \La \to \C$ which take $P_\la$ and $Q_\la$ to $\R_{\geq 0}$---here we recall that we take $q,t \in (-1,1)$. Simply plugging in nonnegative real numbers for the variables in \eqref{eq:finite_cauchy} works, but does not yield all of them. However, a full classification of such homomorphisms, called \emph{Macdonald nonnegative specializations} of $\La$, was conjectured by Kerov \cite{kerov1992generalized} and proven by Matveev \cite{matveev2019macdonald}. We describe them now: they are associated to triples of $\{\alpha_n\}_{n \geq 1}, \{\beta_n\}_{n \geq 1},\tau$ (the Plancherel parameter) such that $\tau \geq 0$, $0 \leq \alpha_n,\beta_n $ for all $ n \geq 1$, and $\sum_n \alpha_n, \sum_n \beta_n < \infty$. These are typically called usual (or alpha) parameters, dual (or beta) parameters, and the Plancherel parameter\footnote{The terminology `Plancherel' comes from the fact that in the case $q=t$ where the Macdonald polynomials reduce to Schur polynomials, the Plancherel specialization is related to (the poissonization of) the Plancherel measure on irreducible representations of the symmetric group $S_N$, see \cite{borodin2017representations}.} respectively. Given such a triple, the corresponding specialization is defined by 
\begin{align}\label{eq:p_specs}
\begin{split}
    p_1 &\mapsto \sum_{n \geq 1} \alpha_n + \frac{1-q}{1-t}\left(\tau + \sum_{n \geq 1} \beta_n\right) \\
    p_k &\mapsto \sum_{n \geq 1} \alpha_n^k + (-1)^{k-1}\frac{1-q^k}{1-t^k}\sum_{n \geq 1} \beta_n^k \quad \quad \text{ for all }k \geq 2.
\end{split}
\end{align}
Note that the above formula defines a specialization for arbitrary tuples of reals $\alpha_n,\beta_n$ and $\tau$ satisfying convergence conditions, but it will not in general be nonnegative.

\begin{defi}\label{def:spec_notation}
For the specialization $\theta$ defined by the triple $\{\alpha_n\}_{n \geq 1}, \{\beta_n\}_{n \geq 1}, \tau$, we write
\begin{equation}\label{eq:spec_argument_notation}
P_\la(\alpha(\alpha_1,\alpha_2,\ldots),\beta(\beta_1,\beta_2,\ldots),\gamma(\tau);q,t) := P_\la(\theta;q,t) := \theta(P_\la)
\end{equation}
and similarly for skew and dual Macdonald polynomials. Likewise, for any other specialization $\phi$ defined by parameters $\{\alpha'_n\}_{n \geq 1}, \{\beta'_n\}_{n \geq 1}, \tau'$, we let
\begin{align}\label{eq:spec_cauchy_kernel}
\begin{split}
   \Pi_{q,t}(\alpha(\alpha_1,\ldots),\beta(\beta_1,\ldots),\gamma(\tau);\alpha(\alpha'_1,\ldots),\beta(\beta'_1,\ldots),\gamma(\tau')) &:= \Pi_{q,t}(\theta;\phi) \\ 
   &:= \exp\left(\sum_{\ell = 1}^\infty \frac{1-t^\ell}{1-q^\ell}\frac{1}{\ell}\theta(p_\ell)\phi(p_\ell)\right). 
\end{split}
\end{align}
We will omit the $\alpha(\cdots)$ in notation if all alpha parameters are zero for the given specialization, and similarly for $\beta$ and $\gamma$.
\end{defi}

We refer to a specialization as 
\begin{itemize}
    \item \emph{pure alpha} if $\tau$ and all $\beta_n, n \geq 1$ are $0$.
    \item \emph{pure beta} if $\tau$ and all $\alpha_n, n \geq 1$ are $0$.
    \item \emph{Plancherel} if all $\alpha_n,\beta_n, n \geq 1$ are $0$.
\end{itemize}

In our specific case, it is useful to define the Plancherel specialization for a general complex parameter as well, though this is \emph{not} a Macdonald-nonnegative specialization.

\begin{defi}\label{def:plancherel}
For any $\tau \in \C$, not necessarily nonnegative real, the corresponding Plancherel specialization of $\Lambda$ is given by
\begin{align}\label{eq:planch_spec}
\begin{split}
    p_1 &\mapsto \frac{1-q}{1-t} \tau \\
    p_k &\mapsto 0 \quad \quad \text{ for all }k \geq 2.
\end{split}
\end{align}
We write $\gamma(\tau)$ in the argument of symmetric functions for this specialization, as in \Cref{def:spec_notation}.
\end{defi}

On Macdonald polynomials, the pure specializations act as follows. 

\begin{prop}\label{thm:specialize_mac_poly}
Let $\la,\mu \in Y$ and $c_1,\ldots,c_n \in \R_{\geq 0}$. Then
\begin{align}\label{eq:spec_mac_pol}
\begin{split}
P_\la(\alpha(c_1,\ldots,c_n);q,t) &= P_\la(c_1,\ldots,c_n;q,t) \\ 
Q_\la(\alpha(c_1,\ldots,c_n);q,t) &= Q_\la(c_1,\ldots,c_n;q,t) \\ 
P_\la(\beta(c_1,\ldots,c_n);q,t) &= Q_{\la'}(c_1,\ldots,c_n;t,q) \\ 
Q_\la(\beta(c_1,\ldots,c_n);q,t) &= P_{\la'}(c_1,\ldots,c_n;t,q),
\end{split}
\end{align}
where in each case the left hand side is a specialized Macdonald symmetric function while the right hand side is a Macdonald polynomial with real numbers plugged in for the variables. Furthermore,
\begin{equation}\label{eq:alpha_gamma_limit}
P_\la(\gamma(\tau);q,t) = \lim_{D \to \infty} P_\la\left(\tau \cdot \frac{1-q}{1-t} \frac{1}{D}[D];q,t\right)
\end{equation}
and similarly for $Q$.
\end{prop}
The alpha case of \eqref{eq:spec_mac_pol}, and \eqref{eq:alpha_gamma_limit}, are straightforward from \eqref{eq:p_specs}. The $\beta$ case follows from properties of a certain involution on $\La$, see \cite[Chapter VI]{mac}, and explains the terminology `dual parameter'.





We note that for any nonnegative specializations $\theta,\phi$ with 
\begin{equation}\label{eq:finiteness_cauchy}
    \sum_{\la \in \Y}P_\la(\theta;q,t)Q_\la(\phi;q,t) < \infty,
\end{equation}
the specialized Cauchy identity
\begin{equation}\label{eq:specialized_cauchy}
    \sum_{\kappa \in \Y} P_{\kappa/\nu}(\theta;q,t)Q_{\kappa/\mu}(\phi;q,t) \\
    =\Pi_{q,t}(\theta;\psi) \sum_{\la \in \Y} Q_{\nu/\la}(\phi;q,t) P_{\mu/\la}(\theta;q,t).
\end{equation}
holds by applying $\theta \ot \phi$ to \eqref{eq:infinite_cauchy}. Similarly, we have the branching rule for specializations: for $\la,\mu \in \Y$, 
\begin{equation}\label{eq:specialization_branch}
P_{\la/\mu}(\phi,\phi';q,t) = \sum_{\nu \in \Y: \mu \subset \nu \subset \la} P_{\la/\nu}(\phi;q,t)P_{\nu/\mu}(\phi';q,t),
\end{equation}
see e.g. \cite[(2.24)]{borodin2014macdonald}. Here by $P_{\la/\mu}(\phi,\phi';q,t)$ we simply mean $\phi(P_{\la/\mu}) + \phi'(P_{\la/\mu})$.

\subsection{Macdonald processes.} One obtains probability measures on sequences of partitions using \eqref{eq:specialized_cauchy} as follows.

\begin{defi}\label{def:mac_proc}
Let $\theta$ and $\phi_1,\ldots,\phi_k$ be Macdonald-nonnegative specializations such that each pair $\theta,\phi_i$ satisfies \eqref{eq:finiteness_cauchy}. The associated \emph{ascending Macdonald process} is the probability measure on sequences $\la^{(1)},\ldots,\la^{(k)}$ given by 
\[
\Pr(\la^{(1)},\ldots,\la^{(k)}) = \frac{Q_{\la^{(1)}}(\phi_1;q,t) Q_{\la^{(2)}/\la^{(1)}}(\phi_2;q,t) \cdots Q_{\la^{(k)}/\la^{(k-1)}}(\phi_k;q,t) P_{\la^{(k)}}(\theta;q,t)}{\prod_{i=1}^k \Pi_{q,t}(\phi_i;\theta)}.
\]
\end{defi}

The $k=1$ case of \Cref{def:mac_proc} is a measure on partitions, referred to as a \emph{Macdonald measure}. Instead of defining joint distributions all at once as above, one can define Markov transition kernels on $\cP$.

\begin{defi}\label{def:cauchy_dynamics}
Let $\theta,\phi$ be Macdonald nonnegative specializations satisfying \eqref{eq:finiteness_cauchy} and $\la$ be such that $P_\la(\theta) \neq 0$. The associated \emph{Cauchy Markov kernel} is defined by 
\begin{equation}\label{eq:def_HL_cauchy_dynamics}
    \Pr(\la \to \nu) = Q_{\nu/\la}(\phi) \frac{P_\nu(\theta)}{P_\la(\theta) \Pi(\phi; \theta)}.
\end{equation}
\end{defi}

It is clear that the ascending Macdonald process above is nothing more than the joint distribution of $k$ steps of a Cauchy Markov kernel with specializations $\phi_i,\theta$ at the $i\tth$ step, and initial condition $\emptyset$. In this work we will only refer to the $q=0$ case, where the Macdonald polynomials are Hall-Littlewood polynomials and the corresponding measure (resp. process) is called a Hall-Littlewood measure (resp. process). 

\section{Explicit formulas for the limit and prelimit}\label{sec:depoissonize}

In this section we collect the algebraic parts of the proof of \Cref{thm:jordan_limit_nometric}. We explicitly define the limiting random variables $\cL_{k,t,\chi}$ below, relate Jordan normal forms of upper-triangular matrices to the Hall-Littlewood process $\la(\tau)$ of \Cref{def:lambda_hl_planch} in \Cref{thm:cite_borodin}, and give the exact formula for the prelimit probabilities we consider in \Cref{thm:depoissonize_int_formula}.

\subsection{The limiting random variable.} Similarly to before, we define $\Sig_\infty := \{(\la_i)_{i \geq 1} \in \Z^\infty: \la_1 \geq \la_2 \geq \ldots\}$.

\begin{defi}\label{def:cL_series}
For any $t \in (0,1),\chi \in \R_{>0}$, we define the $\Sig_\infty$-valued random variable $\cL_{t,\chi} = (\cL_{t,\chi}^{(i)})_{i \geq 1}$ by specifying its finite-dimensional distributions $\cL_{k,t,\chi} := (\cL^{(i)})_{1 \leq i \leq k}$ through the formula
\begin{multline}\label{eq:limit_rv_res_formula_quoted}
\Pr(\cL_{k,t,\chi} = (L_1,\ldots,L_k)) = \frac{1}{(t;t)_\infty}\sum_{d \leq L_k} e^{-\chi t^d} \frac{t^{\sum_{i=1}^k \binom{L_i-d}{2}}}{ (t;t)_{L_k-d} \prod_{i=1}^{k-1}(t;t)_{L_i-L_{i+1}}}  \\ 
\times \sum_{\substack{\mu \in \Sig_{k-1} \\ \mu \prec \vec{L}}} (-1)^{|\vec{L}| - |\mu|-d}  \prod_{i=1}^{k-1} \sqbinom{L_i-L_{i+1}}{L_i-\mu_i}_t Q_{(\mu - (d[k-1]))'}(\gamma((1-t)t^d\chi),\alpha(1);0,t),
\end{multline}
valid for any $\left(L_{1}, \ldots, L_{k}\right) \in \operatorname{Sig}_{k}$.
\end{defi}

It is not obvious that the formulas define valid $\Sig_k$-valued random variables which are consistent for different $k$, but this is shown in \cite[Theorem 6.1]{van2023local}. The following alternative form will be useful for showing convergence to $\cL_{k,t,\chi}$:

\begin{prop}\label{thm:series_to_contour_cL}
The probability weights of $\cL$ have the contour integral expression
\begin{multline}
\label{eq:limit_rv_int_formula_quoted}
\Pr(\cL_{k,t,\chi} = (L_1,\ldots,L_k)) \\ 
=  \frac{(t;t)_\infty^{k-1}}{k! (2 \pi \bi)^k} \prod_{i=1}^{k-1} \frac{t^{\binom{L_i-L_k}{2}}}{(t;t)_{L_i-L_{i+1}}} \int_{\tG^k} e^{\chi t^{L_k}(w_1+\cdots+w_k)} \frac{\prod_{1 \leq i \neq j \leq k} (w_i/w_j;t)_\infty}{\prod_{i=1}^k (-w_i^{-1};t)_\infty (-tw_i;t)_{\infty}} \\ 
\times \sum_{j=0}^{L_{k-1}-L_k} t^{\binom{j+1}{2}} \sqbinom{L_{k-1}-L_k}{j}_t  P_{(L_1-L_k,\ldots,L_{k-1}-L_k,j)}(w_1^{-1},\ldots,w_k^{-1};t,0)  \prod_{i=1}^k \frac{dw_i}{w_i}
\end{multline}
for $k \geq 2$ and 
\begin{equation}
\label{eq:limit_rv_int_formula_k=1_quoted}
\Pr(\cL_{1,t,\chi} = (L)) = \frac{1}{2 \pi \bi} \int_{\tG} \frac{e^{\chi t^L w}}{(-w;t)_\infty} dw
\end{equation}
for $k=1$, with contour 
\begin{equation}
\tG := \{x + \bi: x \leq 0 \} \cup \{x - \bi:  x \leq 0\} \cup \{x+\bi y: x^2+y^2=1, x > 0\}
\end{equation}
in usual counterclockwise orientation, see \Cref{fig:tG}. 
\end{prop}
\begin{proof}
Shown in \cite[Theorem 6.1]{van2023local}.
\end{proof}

\begin{rmk}\label{rmk:k=1_convention}
The $k=1$ formula \eqref{eq:limit_rv_int_formula_k=1_quoted} above appears different from \eqref{eq:limit_rv_int_formula_quoted}, but may be obtained by formally setting $k=1$ and $L_{k-1}-L_k=\infty$ and 
\begin{equation}
\sqbinom{\infty}{j}_t := \frac{1}{(t;t)_j}
\end{equation}
in the general form and applying the $q$-binomial formula.
\end{rmk}

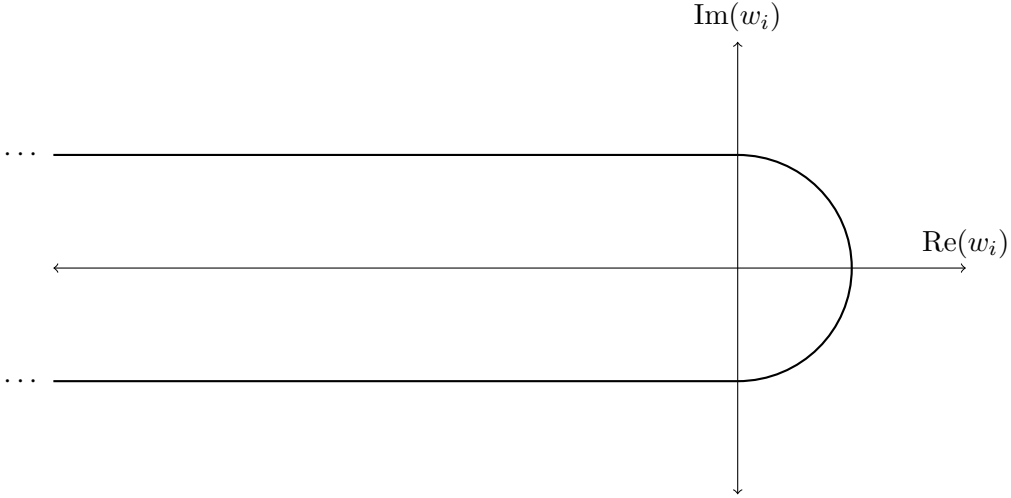
\begin{figure}[htbp]
\begin{center}
\begin{tikzpicture}[scale=1.5]
  \draw[<->] (0,-2) -- (0,2) node[above] {$\Im(w_i)$};
  \draw[<->] (-6,0) -- (2,0) node[above] {$\Re(w_i)$};

  \draw[thick] (0,-1) arc (-90:90:1);

  \draw[thick] (-6,1) node[left] {$\cdots$} -- (0,1); 
  \draw[thick] (-6,-1) node[left] {$\cdots$} -- (0,-1); 

\end{tikzpicture}
\caption{The contour $\tG$ in $\C$.
}\label{fig:tG}
\end{center}
\end{figure}

\subsection{Exact formula for the prelimit probability.} Our analysis proceeds using a specific Hall-Littlewood process related to Jordan forms of triangular matrices, which we now give a name.

\begin{defi}\label{def:lambda_hl_planch}
We denote by $\la(\tau), \tau \in \R_{\geq 0}$ the stochastic process on $\cP$ in continuous time $\tau$ with finite-dimensional marginals given by the Hall-Littlewood process
\begin{equation}\label{eq:hlproc_planch_def}
\Pr(\la(\tau_i) = \la(i) \text{ for all }i=1,\ldots,k) =  \frac{\left(\prod_{j=1}^k Q_{\la(j)/\la(j-1)}(\gamma(\tau_j-\tau_{j-1});0,t) \right) P_{\la(k)}(1,t,\ldots;0,t)}{\exp\left(\frac{\tau_k}{1-t}\right)}
\end{equation}
for each sequence of times $0 \leq \tau_1 \leq \tau_2 \leq \cdots \leq \tau_k$ and $\la(1),\ldots,\la(k)\in \cP$, where in the product we take the convention $\tau_0=0$ and $\la(0) = \emptyset$ is the zero partition. 
\end{defi}

\begin{rmk}\label{rmk:cplx_tau}
For a fixed $\tau \in \C$, we write expectations and probabilities of $\la(\tau)$ to mean with respect to the complex-valued Hall-Littlewood measure defined using the complex Plancherel specialization of \Cref{def:plancherel}, e.g. we use
\begin{equation}
\Pr((\lambda_1'(\tau),\ldots,\lambda_k'(\tau)) = \eta) := \frac{1}{\Pi_{0,t}(1,t,\ldots;\gamma(\tau))}\sum_{\substack{\la \in \Y: \\ (\la_1',\ldots,\la_k') = \eta}} P_\la(1,t,\ldots;0,t)Q_\la(\gamma(\tau);0,t),
\end{equation}
which will in general not be a nonnegative real number, in \Cref{thm:use_torus_product_quoted}.
\end{rmk} 

We then define the stopping times needed to turn it into a discrete-time process.

\begin{defi}\label{def:jump_times}
Define stopping times 
\begin{equation}
\sigma_n := \inf \{\tau \in \R_{\geq 0} : |\la(\tau)| = n\}
\end{equation}
for $n \in \Z_{\geq 0}$.
\end{defi}

Now $\la(\sigma_n), n=0,1,2,\ldots$ is a stochastic process on $\Y$ in discrete time $n$. To compare with random matrices, we require explicit formulas for its transition probabilities.

\begin{prop}\label{thm:explicit_transition_probs}
For any $\mu,\nu \in \Y$,
\begin{equation}\label{eq:disc_la_transitions}
\Pr(\la(\sigma_{n+1}) = \nu | \la(\sigma_n) = \mu) = \begin{cases}
t^{\ell-1}(1-t^{m_{\mu_\ell}(\mu)}) & \exists \ell \geq 1 \text{ such that }\nu_i = \mu_i + \bbone[i=\ell] \text{ for all }i \geq 1 \\ 
0 & \text{else}
\end{cases}
\end{equation}
\end{prop}
\begin{proof}
By \cite{vanpeski2021halllittlewood}\footnote{Specifically, (3.3), (3.4), and (3.6) in the proof of Proposition 3.4, where one must take the parameter $n$ in Proposition 3.4 to be $\infty$.}, $\la(\tau)$ has Markov generator 
\begin{align}
\begin{split}\label{eq:quote_generator}
B(\mu,\nu) &:= \dde{}{\tau}{\tau=0} \Pr(\lambda(T+\tau)=\nu | \lambda(T) = \mu) \\ &= \begin{cases}
-\frac{1}{1-t} & \mu = \nu \\  
\frac{t^{\ell-1}(1-t^{m_{\mu_\ell}(\mu)})}{1-t} & \exists \ell \geq 1 \text{ such that }\nu_i = \mu_i + \bbone[i=\ell] \text{ for all }i \geq 1 \\ 
0 & \text{else}
\end{cases}.
\end{split}
\end{align}
Since
\begin{equation}
\Pr(\la(\sigma_{n+1}) = \nu | \la(\sigma_n) = \mu) = \frac{B(\mu,\nu)}{-B(\mu,\mu)},
\end{equation}
the result follows by \eqref{eq:quote_generator}.
\end{proof}

In the random matrix context, we compare with the following result. Recall that $J(A)$ denotes the partition of sizes of nilpotent Jordan blocks of $A$, arranged in decreasing order.

\begin{prop}[{\cite[Theorem 2.3]{borodin1999lln}}]\label{thm:borodin_division}
Let $(a_{i,j})_{1 \leq i,j \leq n+1}$ be uniformly random strictly upper-triangular over $\F_q$. Then the transition probabilities
\begin{equation}
\Pr(J((a_{i,j})_{1 \leq i,j \leq n+1}) = \nu | (a_{i,j})_{1 \leq i,j \leq n})
\end{equation}
depend only on $J((a_{i,j})_{1 \leq i,j \leq n})$, and 
\begin{multline}
\Pr(J((a_{i,j})_{1 \leq i,j \leq n+1}) = \nu | (a_{i,j})_{1 \leq i,j \leq n} = \mu) \\ 
 = \begin{cases}
q^{-\mu_j'}(1-q^{\mu_j'-\mu_{j-1}'}) & \exists j \geq 1 \text{ such that }\nu_i' = \mu_i' + \bbone[i=j] \text{ for all }i \geq 1 \\ 
0 & \text{else}
\end{cases}
\end{multline}
with the convention $\mu_0'=\infty$ and $q^{-\infty}=0$ when $j=1$.
\end{prop}

Finally, we connect random matrices over $\F_q$ to the Hall-Littlewood process framework.

\begin{cor}\label{thm:cite_borodin}
Let $q$ be a prime power, let $\la(\tau)$ be as in \Cref{def:lambda_hl_planch} with $t=1/q$, and let $\sigma_n, n \in \Z_{\geq 0}$ be as in \Cref{def:jump_times}. Let $a_{i,j}, j > i \geq 1$ be iid uniform elements of $\F_q$, and define $a_{i,j}=0$ for all $i \geq j$. Then for any $N \in \N$,
\begin{equation}\label{eq:dist_eq_jordan}
(\la(\sigma_n))_{0 \leq n \leq N} = (J((a_{i,j})_{1 \leq i,j \leq n}))_{0 \leq n \leq N}
\end{equation}
in distribution.
\end{cor}
\begin{proof}
Both initial conditions are the empty partition, so we must only check that the transition probabilities in \Cref{thm:explicit_transition_probs} and \Cref{thm:borodin_division} are equal. For this simply note that $\ell = \mu_j+1$ and $\mu_{\mu_\ell}'-\mu_{\mu_\ell+1}' = \mu_{j-1}'-\mu_j'$ from our definitions of $\ell$ and $j$ in those results.
\end{proof}

To get formulas for the prelimit probability we are interested in, we first give one for the Poissonized process $\la(\tau)$.

\begin{prop}\label{thm:use_torus_product_quoted}
Let $\la(\tau)$ be as in \Cref{def:lambda_hl_planch}, $k \in \Z_{\geq 1}$, and $\eta \in \Y_k$. Then
\begin{align}\label{eq:use_torus_product_quoted}
\begin{split}
&\Pr((\lambda_1'(\tau),\ldots,\lambda_k'(\tau)) = \eta) = \frac{(t;t)_\infty^{k-1}}{k! (2 \pi \bi)^k  \prod_{i=1}^{k-1} (t;t)_{\eta_i - \eta_{i+1}}} \int_{c\T^k} e^{\frac{ \tau}{1-t}(z_1+\cdots+z_k)}t^{\sum_{i=1}^k \binom{\eta_i}{2}} \\ 
& \times \prod_{1 \leq i \neq j \leq k} (z_i/z_j;t)_\infty \sum_{j=0}^{\eta_{k-1}-\eta_k} t^{j(\eta_k+1)} t^{\binom{j}{2}} \sqbinom{\eta_{k-1}-\eta_k}{j}_t \frac{P_{\eta+j \vec{e_k}}(z_1^{-1},\ldots,z_k^{-1};t,0)}{\prod_{i=1}^k (-z_i^{-1};t)_\infty}  \prod_{i=1}^k \frac{dz_i}{z_i},
\end{split}
\end{align}
where $\T$ denotes the unit circle with counterclockwise orientation, $c \in \R_{>1}$ is arbitrary, and for general complex $\tau$ we interpret the left hand side via \Cref{rmk:cplx_tau}. If $k=1$, we interpret $\eta_{k-1} - \eta_k$ to be $\infty$ and 
\begin{equation}
\sqbinom{\infty}{j}_t := \frac{1}{(t;t)_j}
\end{equation}
in \eqref{eq:use_torus_product_quoted}, c.f. \Cref{thm:series_to_contour_cL} and \Cref{rmk:k=1_convention}.
\end{prop}
\begin{proof}
This is the special case of \cite[Lemma 4.3]{van2023local} where $\nu = \emptyset$.
\end{proof}

We will show the following variant, which was the goal of this subsection, by de-Poissonizing \Cref{thm:use_torus_product_quoted}:

\begin{lemma}\label{thm:depoissonize_int_formula}
Let $\la(\tau)$ be as in \Cref{def:lambda_hl_planch}, let $k, n \in \Z_{\geq 1}$, let $\sigma_n$ be as in \Cref{def:jump_times}, and let $\eta \in \Y_k$. Then
\begin{align}
\begin{split}\label{eq:depoissonize_int_formula}
&\Pr((\lambda_1'(\sigma_n),\ldots,\lambda_k'(\sigma_n)) = \eta) = \frac{(t;t)_\infty^{k-1}}{k! (2 \pi \bi)^k  \prod_{i=1}^{k-1} (t;t)_{\eta_i - \eta_{i+1}}} \int_{c\T^k} (1+z_1+\cdots+z_k)^n t^{\sum_{i=1}^k \binom{\eta_i}{2}} \\ 
&  \times \prod_{1 \leq i \neq j \leq k} (z_i/z_j;t)_\infty \sum_{j=0}^{\eta_{k-1}-\eta_k} t^{j(\eta_k+1)} t^{\binom{j}{2}} \sqbinom{\eta_{k-1}-\eta_k}{j}_t \frac{P_{\eta+j \vec{e_k}}(z_1^{-1},\ldots,z_k^{-1};t,0)}{\prod_{i=1}^k (-z_i^{-1};t)_\infty}  \prod_{i=1}^k \frac{dz_i}{z_i},
\end{split}
\end{align}
where $\T$ denotes the unit circle with counterclockwise orientation, and $c \in \R_{>1}$ is arbitrary. If $k=1$, we interpret $\eta_{k-1} - \eta_k$ to be $\infty$ and 
\begin{equation}
\sqbinom{\infty}{j}_t := \frac{1}{(t;t)_j}
\end{equation}
in \eqref{eq:depoissonize_int_formula}, as in \Cref{thm:use_torus_product_quoted}.
\end{lemma}

\begin{proof}
First note that, because $\la(\sigma_n)$ is independent of $\sigma_n$,
\begin{align}\label{eq:condl_prob_sum}
\begin{split}
\text{LHS\eqref{eq:depoissonize_int_formula}} &= \Pr\left((\la_i'(\tau))_{i \leq i \leq k} = \eta \Big\vert |\la(\tau)|=n\right) \\ 
&= \frac{1}{\Pr(|\la(\tau)|=n)}\sum_{\substack{\la \in \Y: |\la|=n, \\ (\la_i')_{i \leq i \leq k} = \eta}} \Pr(\la(\tau) = \la).
\end{split}
\end{align}
We recall that for any fixed $\la \in \Y$, 
\begin{equation}\label{eq:cplx_tau}
\Pr(\la(\tau) = \la) = e^{-\frac{\tau}{1-t}} P_\la(1,t,\ldots;0,t)Q_\la(\gamma(\tau);0,t).
\end{equation}

The proportionality constant in \eqref{eq:condl_prob_sum} is 
\begin{align}\label{eq:add_integral_circle}
\begin{split}
\Pr(|\la(\tau)|=n) &= \sum_{\substack{\la \in \Y: |\la|=n}} e^{-\frac{\tau}{1-t}} P_\la(1,t,\ldots;0,t)Q_\la(\gamma(\tau);0,t) \\ 
&=  \sum_{\la \in \Y} \frac{1}{2 \pi \bi} \int_\T  e^{-\frac{\tau }{1-t}} P_\la(u,ut,\ldots;0,t)Q_\la(\gamma(\tau);0,t) u^{-n} \frac{du}{u}
\end{split}
\end{align}
where in the second line we use that $P_\la(u,ut,\ldots;0,t) = u^{|\la|}P_\la(1,t,\ldots;0,t)$. The norm of the integrand in \eqref{eq:add_integral_circle} is bounded above by
\begin{equation}
e^{-\frac{\tau}{1-t}} P_\la(1,t,\ldots;0,t)Q_\la(\gamma(\tau);0,t),
\end{equation}
which is summable by the Cauchy identity, so Fubini's theorem applies:
\begin{align}\label{eq:compute_norm_constant}
\begin{split}
\text{RHS\eqref{eq:add_integral_circle}} &= \frac{1}{2 \pi \bi} \int_\T \sum_{\la \in \Y} e^{-\frac{\tau}{1-t}} P_\la(u,ut,\ldots;0,t)Q_\la(\gamma(\tau);0,t) u^{-n} \frac{du}{u} \\ 
&= \frac{1}{2 \pi \bi} \int_\T e^{-\frac{\tau}{1-t}} e^{\frac{u\tau}{1-t}} u^{-n} \frac{du}{u} \\ 
&= \frac{\pfrac{\tau}{1-t}^n}{n! e^{\frac{\tau}{1-t}}}.
\end{split}
\end{align}

The other part of \eqref{eq:condl_prob_sum}, the sum, is
\begin{align}\label{eq:sum_condl_part}
\begin{split}
\sum_{\substack{\la \in \Y: |\la|=n, \\ (\la_i')_{1 \leq i \leq k} = \eta}}\Pr(\la(\tau) = \la) &=  \sum_{\substack{\la \in \Y:  \\ (\la_i')_{1 \leq i \leq k} = \eta}}\frac{e^{-\frac{\tau}{1-t}}}{ 2\pi \bi} \int_\T P_\la(1,t,\ldots;0,t)Q_\la(\gamma(u\tau);0,t) u^{-n} \frac{du}{u} \\ 
&= \frac{e^{-\frac{\tau}{1-t}}}{ 2\pi \bi} \int_\T e^{\frac{u\tau}{1-t}} \sum_{\substack{\la \in \Y: \\ (\la_i')_{1 \leq i \leq k} = \eta}} \frac{P_\la(1,t,\ldots;0,t)Q_\la(\gamma(u\tau);0,t)}{e^{\frac{u\tau}{1-t}}} u^{-n}\frac{du}{u},
\end{split}
\end{align}
where we justify the interchange of sum and integral as previously. For arbitrary $u \in \T$ and $\tau \in \R_{\geq 0}$, \Cref{thm:use_torus_product_quoted} yields
\begin{align}\label{eq:apply_cplx_use_torus}
\begin{split}
&\sum_{\substack{\la \in \Y: \\ (\la_i')_{1 \leq i \leq k} = \eta}} \frac{P_\la(1,t,\ldots;0,t)Q_\la(\gamma(u\tau);0,t)}{\Pi_{0,t}(1,t,\ldots;\gamma(u\tau))}=  \frac{(t;t)_\infty^{k-1}}{k! (2 \pi \bi)^k  \prod_{i=1}^{k-1} (t;t)_{\eta_i - \eta_{i+1}}} \int_{c\T^k} e^{\frac{u \tau}{1-t}(z_1+\cdots+z_k)} \\ 
&  t^{\sum_{i=1}^k \binom{\eta_i}{2}}\prod_{1 \leq i \neq j \leq k} (z_i/z_j;t)_\infty \sum_{j=0}^{\eta_{k-1}-\eta_k} t^{j(\eta_k+1)} t^{\binom{j}{2}} \sqbinom{\eta_{k-1}-\eta_k}{j}_t \frac{P_{\eta+j \vec{e_k}}(z_1^{-1},\ldots,z_k^{-1};t,0)}{\prod_{i=1}^k (-z_i^{-1};t)_\infty}  \prod_{i=1}^k \frac{dz_i}{z_i},
\end{split}
\end{align}
for any $c>1$; here we apply \Cref{thm:use_torus_product_quoted} with $u\tau$ for $\tau$. Substituting \eqref{eq:apply_cplx_use_torus} into \eqref{eq:sum_condl_part} and interchanging the integrals (which are over compact contours so this is immediate) yields
\begin{align}
\begin{split}\label{eq:remove_u_in_sum}
\text{RHS\eqref{eq:sum_condl_part}} &= \frac{(t;t)_\infty^{k-1}}{k! (2 \pi \bi)^k  \prod_{i=1}^{k-1} (t;t)_{\eta_i - \eta_{i+1}}} \int_{c\T^k} \left(\frac{e^{-\frac{\tau}{1-t}}}{ 2\pi \bi}\int_\T  e^{\frac{u\tau}{1-t}} e^{\frac{u \tau}{1-t}(z_1+\cdots+z_k)} u^{-n} \frac{du}{u}  \right)t^{\sum_{i=1}^k \binom{\eta_i}{2}}   \\ 
& \hspace{-10ex}\times \prod_{1 \leq i \neq j \leq k} (z_i/z_j;t)_\infty \sum_{j=0}^{\eta_{k-1}-\eta_k} t^{j(\eta_k+1)} t^{\binom{j}{2}} \sqbinom{\eta_{k-1}-\eta_k}{j}_t \frac{P_{\eta+j \vec{e_k}}(z_1^{-1},\ldots,z_k^{-1};t,0)}{\prod_{i=1}^k (-z_i^{-1};t)_\infty}  \prod_{i=1}^k \frac{dz_i}{z_i} \\ 
&= \frac{(t;t)_\infty^{k-1}}{k! (2 \pi \bi)^k  \prod_{i=1}^{k-1} (t;t)_{\eta_i - \eta_{i+1}}} \int_{c\T^k} \left( \frac{\left(\frac{\tau}{1-t}(1+z_1+\cdots+z_k)\right)^n}{n! e^{\frac{\tau}{1-t}}}\right)t^{\sum_{i=1}^k \binom{\eta_i}{2}} \\ 
&\hspace{-10ex}\times \prod_{1 \leq i \neq j \leq k} (z_i/z_j;t)_\infty \sum_{j=0}^{\eta_{k-1}-\eta_k} t^{j(\eta_k+1)} t^{\binom{j}{2}} \sqbinom{\eta_{k-1}-\eta_k}{j}_t \frac{P_{\eta+j \vec{e_k}}(z_1^{-1},\ldots,z_k^{-1};t,0)}{\prod_{i=1}^k (-z_i^{-1};t)_\infty}  \prod_{i=1}^k \frac{dz_i}{z_i}.
\end{split}
\end{align}
We now evaluate \eqref{eq:condl_prob_sum} by combining \eqref{eq:sum_condl_part}, \eqref{eq:remove_u_in_sum} (for the sum) with \eqref{eq:add_integral_circle}, \eqref{eq:compute_norm_constant} (for the normalizing constant), yielding
\begin{align}
\begin{split}
\text{RHS\eqref{eq:condl_prob_sum}} &= \frac{(t;t)_\infty^{k-1}}{k! (2 \pi \bi)^k  \prod_{i=1}^{k-1} (t;t)_{\eta_i - \eta_{i+1}}} \int_{c\T^k} (1+z_1+\cdots+z_k)^n t^{\sum_{i=1}^k \binom{\eta_i}{2}} \\ 
&\hspace{-10ex}\times \prod_{1 \leq i \neq j \leq k} (z_i/z_j;t)_\infty \sum_{j=0}^{\eta_{k-1}-\eta_k} t^{j(\eta_k+1)} t^{\binom{j}{2}} \sqbinom{\eta_{k-1}-\eta_k}{j}_t \frac{P_{\eta+j \vec{e_k}}(z_1^{-1},\ldots,z_k^{-1};t,0)}{\prod_{i=1}^k (-z_i^{-1};t)_\infty}  \prod_{i=1}^k \frac{dz_i}{z_i}
\end{split}
\end{align}
and completing the proof.
\end{proof}

\begin{rmk}\label{rmk:rederive_poissonized}
Though we do not know a proof of \Cref{thm:depoissonize_int_formula} without passing through \Cref{thm:use_torus_product_quoted}, if one assumes \Cref{thm:depoissonize_int_formula} it is easy to derive \Cref{thm:use_torus_product_quoted} by Poissonizing. Concretely, it follows by the explicit computation of jump rates for $\la(\tau)$ in \eqref{eq:quote_generator} that the size $|\la(\tau)|$ evolves according to a Poisson process in time $\tau$ with rate $1/(1-t)$ (this is also shown algebraically in \eqref{eq:compute_norm_constant} above, but we do not want to use that proof in our argument here because it assumes the result \Cref{thm:use_torus_product_quoted} which we are showing). Hence the two probabilities on the right hand side of 
\begin{align}
\begin{split}
\Pr((\la_1'(\tau),\ldots,\la_k'(\tau)) = \eta) &= \sum_{n \geq 0} \Pr((\la_1'(\tau),\ldots,\la_k'(\tau)) = \eta \Big\vert |\la(\tau)|=n) \Pr(\la(\tau)=n) \\ 
\end{split}
\end{align}
have explicit expressions, by \Cref{thm:depoissonize_int_formula} and the above discussion respectively. Using these and the fact that
\begin{align}
\begin{split}
\sum_{n \geq 0} e^{-\frac{\tau}{1-t}} \frac{\pfrac{\tau}{1-t}^n}{n!} (1+z_1+\cdots+z_k)^n &= e^{-\frac{\tau}{1-t} + \frac{\tau}{1-t}(1+z_1+\cdots+z_k)}= e^{\frac{\tau}{1-t}(z_1+\cdots+z_k)},
\end{split}
\end{align}
\Cref{thm:use_torus_product_quoted} follows.
\end{rmk}

\section{Asymptotic analysis} \label{sec:asymptotics}

In this section we prove \Cref{thm:jordan_limit_intro}. We will actually first prove the equivalent version stated as follows:

\begin{thm}\label{thm:jordan_limit_nometric}
Fix a prime power $q$ and for each $N \geq 1$, let $A_n$ be a uniformly random element of 
\begin{equation}
\mf{g}(n,q) = \{A = (a_{i,j})_{1 \leq i,j \leq n} \in \Mat_n(\F_q): a_{i,j} = 0 \text{ for all } i \geq j \}.
\end{equation}
Let $\zeta \in \R$ and $(n_\ell)_{\ell \geq 1}$ be a subsequence of $\N$ such that $\log_q n_\ell \to -\zeta$ in $\R/\Z$. Then for any $k \geq 1$,
\begin{equation}
(J(A_{n_\ell})'_i - [\log_q n_\ell + \zeta])_{1 \leq i \leq k} \to \cL_{k,q^{-1},q^{-\zeta}}
\end{equation}
in distribution as $\ell \to \infty$, where $\cL_{k,q^{-1},\chi}$ is as defined in \Cref{def:cL_series}, and $[\cdot]$ is the nearest integer function.
\end{thm}

For this, our starting point is the integral formula \Cref{thm:depoissonize_int_formula} for the prelimit probabilities, which we wish to show converges to the corresponding integral formula in \Cref{thm:series_to_contour_cL} for the limit probabilities. For this asymptotic analysis as we carry it out, it is necessary to have a smaller contour than the one in \Cref{thm:depoissonize_int_formula}, which requires moving the contour through several poles. Hence we first compute the corresponding residues in \Cref{thm:compute_a_residue}, and then check they go to $0$ in the limit in \Cref{thm:res_to_0}, before we show the necessary asymptotics of the integral itself. This residue expansion is the only feature where the asymptotic analysis of the de-Poissonized integral in \Cref{thm:depoissonize_int_formula} differs from the analysis of the Poissonized version in \Cref{thm:use_torus_product_quoted}, which was carried out in the proof of \cite[Theorem 4.1]{van2023local}. In that work, it was not necessary residue-expand the integral at all and the appropriate asymptotics could be computed directly.


\begin{lemma}\label{thm:compute_a_residue}
For any $\respow \in \Z_{\geq 0}$ and $k \geq 2$, residue of the integral in \eqref{eq:depoissonize_int_formula} at the pole $z_k = -t^\respow$ is 
\begin{multline}
\label{eq:residue_in_lemma}
E(n,\eta,\respow) :=  \frac{(-1)^\respow (t;t)_\infty^{k-2} t^{\binom{\respow}{2}}t^{\sum_{i=1}^k \binom{\eta_i}{2}}}{k! (t;t)_\respow \prod_{i=1}^{k-1}  (t;t)_{\eta_i - \eta_{i+1}}}\sum_{\substack{\mu \in \Sig_{k-1} \\ \mu \prec \eta}} (-t^{-\respow})^{|\eta|-|\mu|} \prod_{i=1}^{k-1} \sqbinom{\eta_i-\eta_{i+1}}{\eta_i-\mu_i}_t  \\ 
 \times (t^{\eta_k-\respow+1};t)_{\mu_{k-1}-\eta_k}  \frac{1}{(2 \pi \bi)^{k-1}} \int_{\T^{k-1}} (1-t^\respow + z_1+\cdots+z_{k-1})^n  P_\mu(z_1^{-1},\ldots,z_{k-1}^{-1};t,0) \\ 
 \times  \prod_{1 \leq i \neq j \leq k-1} (z_i/z_j;t)_\infty \prod_{1 \leq i \leq k-1} (1+t^{-\respow}z_i) z_i^\respow t^{-\binom{\respow+1}{2}} (-tz_i;t)_\infty \frac{dz_i}{z_i}.
\end{multline}
If $k=1$, then abusing notation and treating $\eta \in \Y_1$ as an integer, the residue is 
\begin{equation}\label{eq:k=1_residue_in_lemma}
E(n,\eta,\respow) := (1-t^\respow)^n \frac{(-1)^{\eta+\respow} t^{\binom{\respow}{2}+\binom{\eta}{2}-\respow \eta}}{(t;t)_{\eta}} \sqbinom{\eta}{\respow}_t.
\end{equation}

\end{lemma}

\begin{proof}
We first show the $k \geq 2$ case, \eqref{eq:residue_in_lemma}. We find the residue of the product term
\begin{equation}\label{eq:product_quotient_term}
\frac{\prod_{1 \leq i \neq j \leq k} (z_i/z_j;t)_\infty}{\prod_{i=1}^k z_i(-z_i^{-1};t)_\infty}
\end{equation}
appearing in \eqref{eq:depoissonize_int_formula}. For the $z_k$-dependent term of the denominator, we see that 
\begin{equation}\label{eq:theta_computation_hidden}
\Res_{z_k=-t^\respow} \frac{1}{z_k(-z_k^{-1};t)_\infty} = \frac{(-1)^\respow t^{\binom{\respow}{2}}}{(t;t)_\respow (t;t)_\infty}.
\end{equation}
When we set $z+k = -t^\respow$, the terms in \eqref{eq:product_quotient_term} which depend on $z_i, i \neq k$ then become
\begin{equation}
\prod_{1 \leq i \leq k-1} \frac{(-z_i/t^\respow;t)_\infty (-t^\respow/z_i;t)_\infty}{z_i (-z_i^{-1};t)_\infty} = \prod_{1 \leq i \leq k-1} (1+t^{-\respow}z_i) z_i^\respow t^{-\binom{\respow+1}{2}} (-tz_i;t)_\infty
\end{equation} 
by an elementary computation, c.f. \cite[(5.17)]{van2023local}, where we let $\binom{\respow+1}{2} = (\respow^2+\respow)/2$ even when $\respow$ is negative. By the branching rule (\eqref{eq:def_skewP} and \Cref{thm:hl_qw_branch_formulas}), the sum over $j$ with the $q$-Whittaker polynomial in \eqref{eq:depoissonize_int_formula} becomes
\begin{align}
\begin{split}\label{eq:branch_cancel_j}
&\sum_{j=0}^{\eta_{k-1}-\eta_k} t^{\binom{j}{2}+(\eta_k+1)j} \sqbinom{\eta_{k-1}-\eta_k}{j}_t P_{\eta + j \vec{e}_k}(z_1^{-1},\ldots,z_{k-1}^{-1},-t^{-\respow};t,0) \\ 
&= \sum_{j=0}^{\eta_{k-1}-\eta_k} t^{\binom{j}{2}+(\eta_k+1)j} \sqbinom{\eta_{k-1}-\eta_k}{j}_t \sum_{\substack{\mu \in \Sig_{k-1} \\ \mu \prec \eta+j\vec{e}_k}} (-t^{-\respow})^{|\eta|+j-|\mu|} \sqbinom{\eta_1-\eta_2}{\eta_1-\mu_1}_t \cdots \sqbinom{\eta_{k-2}-\eta_{k-1}}{\eta_{k-2} - \mu_{k-2}}_t \\ 
&\times \sqbinom{\eta_{k-1}-\eta_k-j}{\eta_{k-1}-\mu_{k-1}}_t P_\mu(z_1^{-1},\ldots,z_{k-1}^{-1};t,0) \\ 
&= \sum_{\substack{\mu \in \Sig_{k-1} \\ \mu \prec \eta}}(-t^{-\respow})^{|\eta|-|\mu|} \sqbinom{\eta_1-\eta_2}{\eta_1-\mu_1}_t \cdots \sqbinom{\eta_{k-2}-\eta_{k-1}}{\eta_{k-2} - \mu_{k-2}}_t \sqbinom{\eta_{k-1}-\eta_{k}}{\eta_{k-1} - \mu_{k-1}}_t P_\mu(z_1^{-1},\ldots,z_{k-1}^{-1};t,0) \\ 
&\times \sum_{j=0}^{\mu_{k-1} - \eta_k} \sqbinom{\mu_{k-1} - \eta_k}{j}_t  t^{\binom{j}{2}+(\eta_k+1)j}(-t^{-\respow})^j \\ 
&= \sum_{\substack{\mu \in \Sig_{k-1} \\ \mu \prec \eta}}(-t^{-\respow})^{|\eta|-|\mu|} \prod_{i=1}^{k-1} \sqbinom{\eta_i-\eta_{i+1}}{\eta_i - \mu_i}_t P_\mu(z_1^{-1},\ldots,z_{k-1}^{-1};t,0) (t^{\eta_k-\respow+1};t)_{\mu_{k-1}-\eta_k},
\end{split}
\end{align}
where in the second equality we used the elementary identity
\begin{equation}
\label{eq:apply_qbin_swap}
\sqbinom{\eta_{k-1}-\eta_k-j}{\eta_{k-1}-\mu_{k-1}}_t \sqbinom{\eta_{k-1}-\eta_k}{j}_t = \sqbinom{\mu_{k-1}-\eta_k}{j}_t \sqbinom{\eta_{k-1}-\eta_k}{\eta_{k-1}-\mu_{k-1}}_t
\end{equation}
and in the last equality used the $q$-binomial theorem. Putting this all together, the residue is 
 \begin{multline}\label{eq:residue_in_lemma_proof}
 \frac{(-1)^\respow (t;t)_\infty^{k-2} t^{\binom{\respow}{2}}t^{\sum_{i=1}^k \binom{\eta_i}{2}}}{k! (t;t)_\respow \prod_{i=1}^{k-1}  (t;t)_{\eta_i - \eta_{i+1}}}\sum_{\substack{\mu \in \Sig_{k-1} \\ \mu \prec \eta}} (-t^{-\respow})^{|\eta|-|\mu|} \prod_{i=1}^{k-1} \sqbinom{\eta_i-\eta_{i+1}}{\eta_i-\mu_i}_t  (t^{\eta_k-\respow+1};t)_{\mu_{k-1}-\eta_k} \\ 
 \times \frac{1}{(2 \pi \bi)^{k-1}} \int_{\T^{k-1}} (1-t^\respow + z_1+\cdots+z_{k-1})^n  P_\mu(z_1^{-1},\ldots,z_{k-1}^{-1};t,0) \\ 
 \times  \prod_{1 \leq i \neq j \leq k-1} (z_i/z_j;t)_\infty \prod_{1 \leq i \leq k-1} (1+t^{-\respow}z_i) z_i^\respow t^{-\binom{\respow+1}{2}} (-tz_i;t)_\infty \frac{dz_i}{z_i},
 \end{multline}
where we have also shifted the contour from $c \T^{k-1}$ to $\T^{k-1}$, which we may do because there are no poles. This completes the $k \geq 2$ case.

For $k=1$, we simply note that $P_{(\eta_k+j)}(z_1^{-1};t,0) = z_1^{-(\eta_k+j)}$, and
\begin{align}
\begin{split}
\sum_{j  = 0}^\infty t^{j(\eta_k+1)+\binom{j}{2}} \frac{z_1^{-(\eta_k+j)}}{(t;t)_j (-z_1^{-1};t)_\infty} &= z_1^{-\eta_k} \frac{(-t^{\eta_k+1}z_1^{-1};t)_\infty}{(-z_1^{-1};t)_\infty}
\end{split}
\end{align}
by the $q$-binomial theorem. By \eqref{eq:theta_computation_hidden}, we therefore have
\begin{align}
\begin{split}
E(n,\eta,\respow) &= \Res_{z_1=-t^\respow} (1+z_1)^n t^{\binom{\eta_k}{2}} z_1^{-\eta_k} \frac{(-t^{\eta_k+1}z_1^{-1};t)_\infty}{z_1(-z_1^{-1};t)_\infty} \\ 
&= (1-t^\respow)^n \frac{(-1)^{\eta_k+\respow} t^{\binom{\respow}{2}+\binom{\eta_k}{2}-\respow \eta_k}}{(t;t)_{\eta_k}} \sqbinom{\eta_k}{\respow}_t.
\end{split}
\end{align}
\end{proof}

\begin{lemma}\label{thm:res_to_0}
Let $k \in \Z_{\geq 1}$ and $E(n,\eta,\respow)$ be as in \eqref{eq:residue_in_lemma} or \eqref{eq:k=1_residue_in_lemma}. Fix $(L_1,\ldots,L_k) \in \Sig_k$, and define 
\begin{equation}
\eta(n) := (L_i + [\log_{t^{-1}} n+\zeta])_{1 \leq i \leq k)}
\end{equation}
where $[\cdot]$ is the nearest integer function as usual. Then for any $\respow \in \Z_{\geq 0}$,
\begin{equation}\label{eq:E_dies}
\lim_{n \to \infty} E(n,\eta(n),\respow) = 0.
\end{equation}
\end{lemma}
\begin{proof}
The case $k=1$ is trivial from \Cref{thm:compute_a_residue}, since the $(1-t^\respow)^n$ term dominates in \eqref{eq:k=1_residue_in_lemma}, so let us suppose $k \geq 2$. The integrands in $E(n,\eta(n),\respow)$ have no poles on the punctured plane, so we first shrink the contour to $c\T$ where $c > 0$ is such that $1-t^\respow+c(k-1) < 1$. We now rewrite the formula for $E(n,\eta(n),\respow)$ as follows. Define $\tLL := (L_1-L_k,\ldots,L_{k-1}-L_k,0)$ and, for each $\mu$ in the sum in \Cref{thm:compute_a_residue} (with $\eta(n)$ substituted for $\eta$), let $\tmu := (\mu_i - \eta_k(n))_{1 \leq i \leq k-1}$. Then by \Cref{thm:signature_shift},
\begin{equation}
P_\mu(z_1^{-1},\ldots,z_{k-1}^{-1};t,0) = (z_1 \cdots z_{k-1})^{-\eta_k(n)}P_{\tmu}(z_1^{-1},\ldots,z_{k-1}^{-1};t,0).
\end{equation}
Substituting this and moving all $n$-dependent terms inside the integral, we have
\begin{multline}\label{eq:massaged_integral_about_to_die}
E(n,\eta(n),\respow) = \frac{(-1)^\respow (t;t)_\infty^{k-2} t^{\binom{\respow}{2}}}{k! (t;t)_\respow \prod_{i=1}^{k-1}  (t;t)_{\tLL_i - \tLL_{i+1}}}\sum_{\substack{\tmu \in \Sig_{k-1} \\ \tmu \prec \tLL}} (-t^{-\respow})^{|\tLL|-|\tmu|} \prod_{i=1}^{k-1} \sqbinom{\tLL_i - \tLL_{i+1}}{\tLL_i - \tmu_i}_t  \\ 
 \times \frac{1}{(2 \pi \bi)^{k-1}} \int_{\T^{k-1}} \left[(t^{\eta_k(n)-\respow+1};t)_{\tmu_{k-1}}(1-t^\respow + z_1+\cdots+z_{k-1})^n (-t^{-\respow})^{\eta_k(n)}t^{\sum_{i=1}^k \binom{\eta_i}{2}}\right]   \\ 
 \times  P_\mu(z_1^{-1},\ldots,z_{k-1}^{-1};t,0)\prod_{1 \leq i \neq j \leq k-1} (z_i/z_j;t)_\infty \prod_{1 \leq i \leq k-1} (1+t^{-\respow}z_i) z_i^\respow t^{-\binom{\respow+1}{2}} (-tz_i;t)_\infty \frac{dz_i}{z_i}.
\end{multline}
All $n$-dependent terms in \eqref{eq:massaged_integral_about_to_die} are inside the square braces, and the modulus of the term in braces is 
\begin{equation}
O\left(e^{\log(1-t^\respow+c(k-1))n + \text{const}_1 \eta_k(n)^2 + \text{const}_2 \eta_k(n)}\right).
\end{equation}
Since $\eta_k(n) = L_k + [\log_{t^{-1}} n+\zeta]$, the $\log(1-t^\respow+c(k-1))n$ term dominates, hence the function inside the exponential goes to $-\infty$ since we took $c$ such that $1-t^\respow+c(k-1) < 1$. Because the sum over $\tmu$ is finite independent of $n$ and the contour of integration is compact and independent of $n$, we may bring the limit in \eqref{eq:E_dies} inside the integral, and this completes the proof.
\end{proof}

\begin{proof}[Proof of \Cref{thm:jordan_limit_nometric}]
We set $t=1/q$ throughout the proof. Since $\cL_{k,\cdot,\cdot}$ is a discrete random variable, it suffices to show 
\begin{equation}\label{eq:main_thm_wts}
\lim_{\ell \to \infty} \Pr((J(A_{n_\ell})_i')_{1 \leq i \leq k} = (L_i + [\log_{t^{-1}} n_\ell + \zeta])_{1 \leq i \leq k)}) = \Pr(\cL_{k,t,t^\zeta} = (L_1,\ldots,L_k)).
\end{equation}
Let us define 
\begin{equation}
\eta(n) := (L_i + [\log_{t^{-1}} n + \zeta])_{1 \leq i \leq k)}
\end{equation}
as in \Cref{thm:res_to_0}, and note that
\begin{equation}
\label{eq:eta_difference}
\eta_i(n)-\eta_j(n) = L_i - L_j.
\end{equation}
By reexpressing the left hand side of \eqref{eq:main_thm_wts} via \Cref{thm:cite_borodin} and \Cref{thm:depoissonize_int_formula}, and the right hand side via \Cref{thm:series_to_contour_cL}, we see that \eqref{eq:main_thm_wts} is equivalent to
\begin{multline}
\label{eq:explicit_main_thm_wts}
\lim_{\ell \to \infty} \frac{(t;t)_\infty^{k-1}}{k! (2 \pi \bi)^k  \prod_{i=1}^{k-1} (t;t)_{\eta(n_\ell)_i - \eta(n_\ell)_{i+1}}} \int_{c\T^k} (1+z_1+\cdots+z_k)^{n_\ell} t^{\sum_{i=1}^k \binom{\eta({n_\ell})_i}{2}} \prod_{1 \leq i \neq j \leq k} (z_i/z_j;t)_\infty \\ 
  \times \sum_{j=0}^{\eta({n_\ell})_{k-1}-\eta({n_\ell})_k} t^{j(\eta({n_\ell})_k+1)} t^{\binom{j}{2}} \sqbinom{\eta({n_\ell})_{k-1}-\eta({n_\ell})_k}{j}_t \frac{P_{\eta({n_\ell})+j \vec{e_k}}(z_1^{-1},\ldots,z_k^{-1};t,0)}{\prod_{i=1}^k (-z_i^{-1};t)_\infty}  \prod_{i=1}^k \frac{dz_i}{z_i} \\ 
= \frac{(t;t)_\infty^{k-1}}{k! (2 \pi \bi)^k} \prod_{i=1}^{k-1} \frac{t^{\binom{L_i-L_k}{2}}}{(t;t)_{L_i-L_{i+1}}} \int_{\tG^k} e^{t^{L_k+\zeta}(w_1+\cdots+w_k)} \frac{\prod_{1 \leq i \neq j \leq k} (w_i/w_j;t)_\infty}{\prod_{i=1}^k (-w_i^{-1};t)_\infty (-tw_i;t)_{\infty}} \\ 
\times \sum_{j=0}^{L_{k-1}-L_k} t^{\binom{j+1}{2}} \sqbinom{L_{k-1}-L_k}{j}_t  P_{(L_1-L_k,\ldots,L_{k-1}-L_k,j)}(w_1^{-1},\ldots,w_k^{-1};t,0)  \prod_{i=1}^k \frac{dw_i}{w_i}.
\end{multline}
The integrand in \eqref{eq:explicit_main_thm_wts} is symmetric in $z_1,\ldots,z_k$ and has residues at $z_i=-t^0,-t^1,\ldots$. For the rest of the proof, fix $r$ to be any positive integer such that $t^r < 1/k$. Shrinking each contour to $t^{r+1/2}\T^k$, we encounter the pole at $z_i = -t^0,\ldots,-t^{r}$ for each $i=1,\ldots,k$. The corresponding residues were computed for $z_k$ in \Cref{thm:compute_a_residue} and are the same for each $z_i$ by symmetry, so we obtain
\begin{multline}\label{eq:partial_residue_expansion}
\text{LHS\eqref{eq:explicit_main_thm_wts}} = \lim_{\ell \to \infty} k \sum_{\respow = 0}^{r} E(n_\ell,\eta(n_\ell),\respow) + \frac{(t;t)_\infty^{k-1}}{k! (2 \pi \bi)^k  \prod_{i=1}^{k-1} (t;t)_{\eta_i - \eta_{i+1}}} \int_{t^{r+1/2}\T^k} (1+z_1+\cdots+z_k)^{n_\ell}  \\ 
  \times t^{\sum_{i=1}^k \binom{\eta_i}{2}}\prod_{1 \leq i \neq j \leq k} (z_i/z_j;t)_\infty \sum_{j=0}^{\eta_{k-1}-\eta_k} t^{j(\eta_k+1)} t^{\binom{j}{2}} \sqbinom{\eta_{k-1}-\eta_k}{j}_t \frac{P_{\eta+j \vec{e_k}}(z_1^{-1},\ldots,z_k^{-1};t,0)}{\prod_{i=1}^k (-z_i^{-1};t)_\infty}  \prod_{i=1}^k \frac{dz_i}{z_i}.
\end{multline}
The utility of the residue expansion we have just carried out is that now the integral is over a smaller contour, allowing later error analysis which would break down if the real parts of $z_i$ became too negative. We will now algebraically manipulate the integral in \eqref{eq:partial_residue_expansion} into a form more suitable for asymptotics, and to control subscripts we do so with $n$ in place of $n_\ell$. Let
\begin{equation}
\label{eq:def_eps_error}
\eps(n) := (\log_{t^{-1}} (n) + \zeta) - [\log_{t^{-1}}(n) + \zeta],
\end{equation}
so that $\lim_{\ell \to \infty} \eps(n_\ell) = 0$ by our choice of the subsequence $n_\ell$. We then make a change of variables to
\begin{equation}
w_i := t^{-\eta_k(n)}z_i = n t^{-\zeta - L_k + \eps(n)}z_i.
\end{equation}
Note that for $n=n_\ell$ this is $n_\ell t^{-\zeta - L_k}z_i (1+o(1))$ as $\ell \to \infty$, but it will be clearer later to have the $t^{\eps(n_\ell)} = 1+o(1)$ multiplicative error term written explicitly, as the rate at which $\eps(n_\ell)$ goes to $0$ influences our choice of contours later.

With this change of variables, the integral on the left hand side of \eqref{eq:int_for_asymptotics} becomes
\begin{align}
\label{eq:int_var_change}
\begin{split}
&\frac{(t;t)_\infty^{k-1}}{k! (2 \pi \bi)^k  \prod_{i=1}^{k-1} (t;t)_{L_i - L_{i+1}}} \int_{t^{-\eta_k(n)+r+1/2}\T^k} \left(1+\frac{t^{\zeta+L_k-\eps(n)}}{n}(w_1+\cdots+w_k))\right)^n \\ 
&\times \sum_{j=0}^{L_{k-1}-L_k} t^{j(\eta_k(n)+1)+\binom{j}{2}} \sqbinom{L_{k-1}-L_k}{j}_t (t^{-\eta_k(n)})^{|\eta(n)|+j} P_{\eta(n)+j \vec{e_k}}(w_1^{-1},\ldots,w_k^{-1};t,0) \\ 
&\times  \frac{\prod_{1 \leq i \neq j \leq k} (w_i/w_j;t)_\infty \prod_{i=1}^k t^{\binom{\eta_i(n)}{2}}}{{\prod_{i=1}^k (-t^{-\eta_k(n)}w_i^{-1};t)_\infty}}  \prod_{i=1}^k \frac{dw_i}{w_i},
\end{split}
\end{align}
where we have used that $P_{\eta+j\vec{e_k}}$ is homogeneous of degree $|\eta|+j$. By the elementary identity 
\begin{equation}
\label{eq:favorite_binomial_split}
\binom{a+b}{2} = \binom{a}{2} + \binom{b}{2} + ab
\end{equation}
and \eqref{eq:eta_difference} we have 
\begin{equation}\label{eq:eta_i_to_k}
\binom{\eta_i}{2} = \binom{\eta_k}{2} + \binom{L_i-L_k}{2} + (\eta_i-\eta_k)\eta_k.
\end{equation}
Additionally, by \Cref{thm:signature_shift} and \eqref{eq:eta_difference}, 
\begin{equation}
\label{eq:shift_by_tau}
P_{\eta+j \vec{e_k}}(w_1^{-1},\ldots,w_k^{-1};t,0) = (w_1 \cdots w_k)^{-\eta_k}P_{(L_1-L_k,\ldots,L_{k-1}-L_k,j)}(w_1^{-1},\ldots,w_k^{-1};t,0).
\end{equation}
Substituting \eqref{eq:eta_i_to_k} for $1 \leq i \leq k$ and \eqref{eq:shift_by_tau} into \eqref{eq:int_var_change} yields
\begin{align}
\label{eq:int_w_2}
\begin{split}
&\text{RHS\eqref{eq:int_var_change}} = \frac{(t;t)_\infty^{k-1}}{k! (2 \pi \bi)^k} \prod_{i=1}^{k-1} \frac{t^{\binom{L_i-L_k}{2}}}{(t;t)_{L_i-L_{i+1}}} \int_{t^{-\eta_k(n)+r+1/2}\T^k}  \left(1+\frac{t^{\zeta+L_k-\eps(n)}}{n}(w_1+\cdots+w_k))\right)^n \\ 
&\times  \prod_{1 \leq i \neq j \leq k} (w_i/w_j;t)_\infty  \prod_{i=1}^k \frac{w_i^{-\eta_k(n)}t^{\binom{\eta_k(n)}{2}+(\eta_i(n)-\eta_k(n))\eta_k(n)}t^{-\eta_k(n)\eta_i(n)}}{{(-t^{-\eta_k(n)}w_i^{-1};t)_\infty}} \\ 
&\times \sum_{j=0}^{L_{k-1}-L_k} t^{j+\binom{j}{2}} \sqbinom{L_{k-1}-L_k}{j}_t P_{(L_1-L_k,\ldots,L_{k-1}-L_k,j)}(w_1^{-1},\ldots,w_k^{-1};t,0)  \prod_{i=1}^k \frac{dw_i}{w_i}.
\end{split}
\end{align}
Noting that 
\begin{equation}
\frac{w_i^{-\eta_k(n)}t^{\binom{\eta_k(n)}{2}+(\eta_i(n)-\eta_k(n))\eta_k(n)}t^{-\eta_k(n)\eta_i(n)}}{{(-t^{-\eta_k(n)}w_i^{-1};t)_\infty}} = \frac{1}{(-w_i^{-1};t)_\infty (-tw_i;t)_{\eta_k(n)}},
\end{equation}
and shifting contours to
\begin{multline}
\Gamma(r,n) := \{x+\bi y: x^2+y^2=1, x > 0\}  \cup \{x + \bi: -t^{-\eta_k(n)+r+1/2} < x \leq 0 \} \\ \cup \{x - \bi: -t^{-\eta_k(n)+r+1/2} < x \leq 0 \} \cup \{-t^{-\eta_k(n)+r+1/2} + \bi y: -1 \leq y \leq 1\}
\end{multline}
(see \Cref{fig:Gamma1_decomp}), yields
\begin{align}
\label{eq:int_w_3}
\begin{split}
&\text{RHS\eqref{eq:int_w_2}} = \frac{(t;t)_\infty^{k-1}}{k! (2 \pi \bi)^k} \prod_{i=1}^{k-1} \frac{t^{\binom{L_i-L_k}{2}}}{(t;t)_{L_i-L_{i+1}}} \int_{\Gamma(r,n)^k} \left(1+\frac{t^{\zeta+L_k-\eps(n)}}{n}(w_1+\cdots+w_k)\right)^n \\ 
&\times \frac{\prod_{1 \leq i \neq j \leq k} (w_i/w_j;t)_\infty}{\prod_{i=1}^k (-w_i^{-1};t)_\infty (-tw_i;t)_{\eta_k(n)}} \sum_{j=0}^{L_{k-1}-L_k} t^{\binom{j+1}{2}} \sqbinom{L_{k-1}-L_k}{j}_t  \\ 
&\times P_{(L_1-L_k,\ldots,L_{k-1}-L_k,j)}(w_1^{-1},\ldots,w_k^{-1};t,0)  \prod_{i=1}^k \frac{dw_i}{w_i}.
\end{split}
\end{align}
For the asymptotics, we will decompose the integration contour into a main term contour $\Gamma_1(r,n)$ and error term contour $\Gamma_2(r,n)$. Let $\xi(n)$ be any function such that
\begin{enumerate}
\item $\xi(n) \to \infty$ as $n \to \infty$,
\item $\xi(n) \ll \log_{t^{-1}} n$ as $n \to \infty$, and
\item $\xi(n_\ell) \ll -\log \eps(n_\ell)$ as $\ell \to \infty$.
\end{enumerate}
Any sufficiently slowly-growing function will do, but what we mean by `slowly-growing' depends on how quickly the shifts $\log_t n_\ell$ become close to the appropriate lattice, and this is why we gave the error term $\eps(n_\ell)$ a name earlier.

Now decompose $\Gamma(r,n)$ as
\begin{align}\label{eq:contours}
\begin{split}
\Gamma(r,n) &= \Gamma_1(n) \cup \Gamma_2(r,n) \\ 
\Gamma_1(n) &= \{x + \bi: -t^{-\xi(n)} < x \leq 0 \} \cup \{x - \bi: -t^{-\xi(n)}  < x \leq 0\} \cup \{x+\bi y: x^2+y^2=1, x > 0\} \\ 
\Gamma_2(r,n) &= \{-t^{-\eta_k(n)+r-1/2} + \bi y: -1 \leq y \leq 1\} \cup \{x + \bi: -t^{-\eta_k(n)+r-1/2}< x \leq  -t^{-\xi(n)}  \} \\ 
&  \cup \{x - \bi: -t^{-\eta_k(n)+r-1/2}< x \leq  -t^{-\xi(n)} \}.
\end{split}
\end{align}
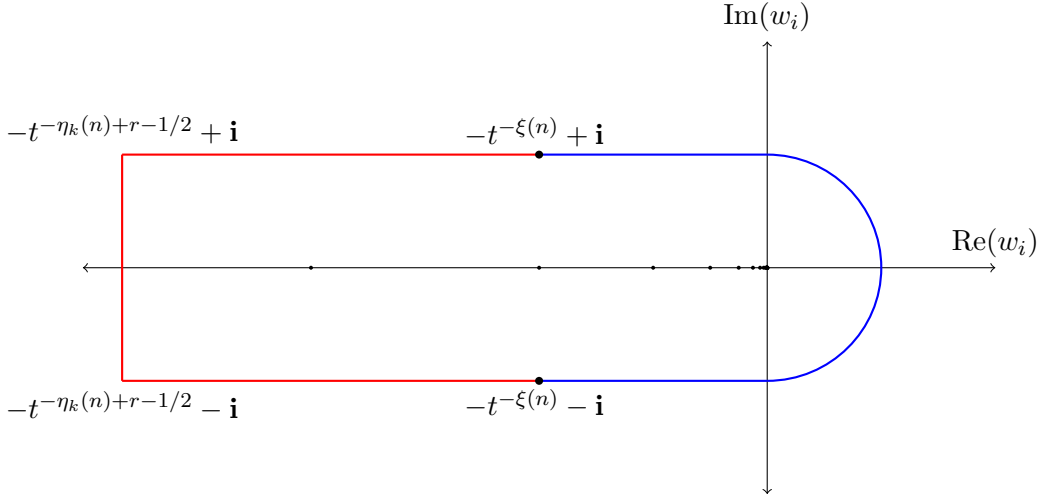
\begin{figure}[htbp]
\begin{center}
\begin{tikzpicture}[scale=1.5]
\def\b{2};
\def\t{.5};
\def\d{-5.656}; 
  \draw[<->] (0,-2) -- (0,2) node[above] {$\Im(w_i)$};
  \draw[<->] (-6,0) -- (2,0) node[above] {$\Re(w_i)$};

  \draw[thick,blue] (0,-1) arc (-90:90:1);

  \draw[thick,blue] (-\b,1) node[left,black,yshift=3mm,xshift=10mm] {$-t^{-\xi(n)}+\bi$} -- (0,1) ;
  \draw[thick,blue] (-\b,-1) node[left,black,yshift=-3mm,xshift=10mm] {$-t^{-\xi(n)}-\bi$} -- (0,-1) ;

    \draw[thick,red] (\d,1)  -- (-\b,1);
    \draw[thick,red] (\d,-1)  -- (-\b,-1);

    \draw[thick,red] (\d,1) node[above,black] {$-t^{-\eta_k(n)+r-1/2}+\bi$} -- (\d,-1) node[below,black] {$-t^{-\eta_k(n)+r-1/2}-\bi$};

\fill (-\b,1) circle (1pt);
\fill (-\b,-1) circle (1pt);

  \foreach \n in {-2,...,20} {
    \pgfmathsetmacro\x{-pow(\t,\n)}
    \fill (\x,0) circle (1pt);
  }

\end{tikzpicture}
\caption{The contour $\Gamma(r,n)$ decomposed as in \eqref{eq:contours}, with $\Gamma_1(n)$ in blue and $\Gamma_2(r,n)$ in red, and the poles of the integrand at $w_i = -t^\Z$ shown.
}\label{fig:Gamma1_decomp}
\end{center}
\end{figure}

We further define another error term contour
\begin{equation}
\Gamma_3(n) = \{x + \bi:x \leq  -t^{- \xi(n)}\} \cup \{x - \bi: x \leq -t^{- \xi(n)}  \},
\end{equation}
so that 
\begin{equation}\label{eq:tG_decomp}
\Gamma_1(n) \cup \Gamma_3(n) = \{x + \bi: x \leq 0 \} \cup \{x - \bi:  x \leq 0\} \cup \{x+\bi y: x^2+y^2=1, x > 0\} = \tG
\end{equation}
\begin{figure}[htbp]
\begin{center}
\begin{tikzpicture}[scale=1.5]
\def\b{2}
  \draw[<->] (0,-2) -- (0,2) node[above] {$\Im(w_i)$};
  \draw[<->] (-6,0) -- (2,0) node[above] {$\Re(w_i)$};

  \draw[thick,blue] (0,-1) arc (-90:90:1);

  \draw[thick,blue] (-\b,1) node[left,black,yshift=3mm,xshift=10mm] {$-t^{-\xi(n)}+\bi$} -- (0,1) ;
  \draw[thick,blue] (-\b,-1) node[left,black,yshift=-3mm,xshift=10mm] {$-t^{-\xi(n)}-\bi$} -- (0,-1) ;

    \draw[thick,green] (-6,1) node[left,black] {$\cdots$} -- (-\b,1);
    \draw[thick,green] (-6,-1) node[left,black] {$\cdots$} -- (-\b,-1);

\fill (-\b,1) circle (1pt);
\fill (-\b,-1) circle (1pt);

\end{tikzpicture}
\caption{The contour $\tG$ decomposed as in \eqref{eq:tG_decomp}, with $\Gamma_1(n)$ in blue and $\Gamma_3(n)$ in green.
}\label{fig:tG_decomp}
\end{center}
\end{figure}
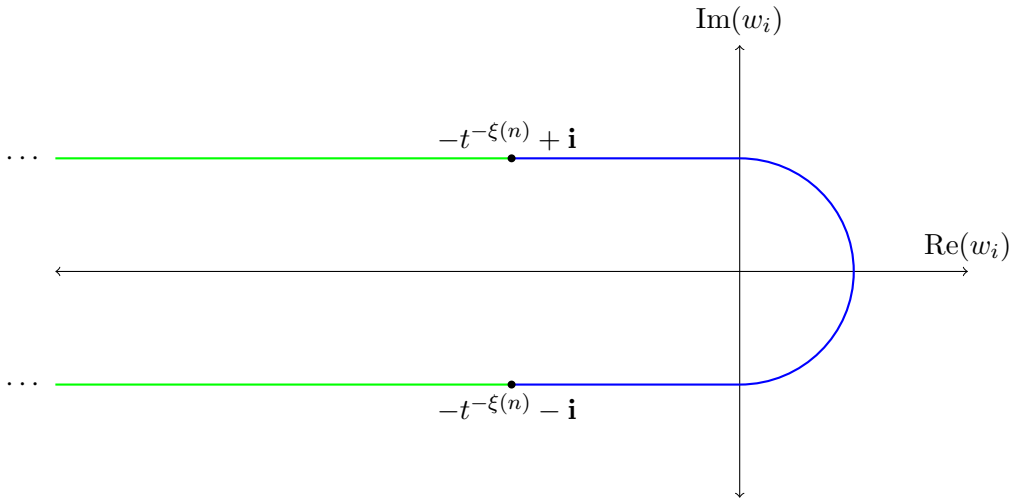
is independent of $n$. 

To compress notation we abbreviate the $n$-independent part of the integrand as
\begin{multline}
\label{eq:integrand_junk}
\tf(w_1,\ldots,w_k) := \frac{(t;t)_\infty^{k-1}}{k!} \prod_{i=1}^{k-1} \frac{t^{\binom{L_i-L_k}{2}}}{(t;t)_{L_i-L_{i+1}}} \frac{\prod_{1 \leq i \neq j \leq k} (w_i/w_j;t)_\infty}{\prod_{i=1}^k w_i(-w_i^{-1};t)_\infty }  \\ 
\times \sum_{j=0}^{L_{k-1}-L_k} t^{\binom{j+1}{2}} \sqbinom{L_{k-1}-L_k}{j}_t  P_{(L_1-L_k,\ldots,L_{k-1}-L_k,j)}(w_1^{-1},\ldots,w_k^{-1};t,0).
\end{multline}
When $k=1$, we interpret the above by setting $L_{k-1} - L_k = \infty$ and $\sqbinom{\infty}{j}_t = \frac{1}{(t;t)_j}$, similarly to \Cref{thm:series_to_contour_cL}. We have shown by the above manipulations that
\begin{multline}
\text{LHS\eqref{eq:explicit_main_thm_wts}} = \lim_{\ell \to \infty} k \sum_{\respow=0}^r E({n_\ell},\eta({n_\ell}),\respow) \\ 
+ \lim_{\ell \to \infty} \frac{1}{(2 \pi \bi)^k} \int_{\Gamma({n_\ell},r)^k} \tf(w_1,\ldots,w_k) \frac{\left(1+\frac{t^{\zeta+L_k-\eps(n_\ell)}}{{n_\ell}}(w_1+\cdots+w_k)\right)^{n_\ell}}{\prod_{i=1}^k (-tw_i;t)_{\eta_k({n_\ell})}} \prod_{i=1}^k dw_i.
\end{multline}
The limit of the residues above is $0$ by \Cref{thm:res_to_0} since $r$ is independent of $\ell$, so to show \eqref{eq:explicit_main_thm_wts} it suffices to show 
\begin{equation}\label{eq:int_for_asymptotics}
\lim_{\ell \to \infty} \frac{1}{(2 \pi \bi)^k} \int_{\Gamma({n_\ell},r)^k} \tf(w_1,\ldots,w_k) \frac{\left(1+\frac{t^{\zeta+L_k-\eps(n_\ell)}}{{n_\ell}}(w_1+\cdots+w_k)\right)^{n_\ell}}{\prod_{i=1}^k (-tw_i;t)_{\eta_k({n_\ell})}} \prod_{i=1}^k dw_i = \text{RHS\eqref{eq:explicit_main_thm_wts}}.
\end{equation}
By subtracting the right hand side of \eqref{eq:int_for_asymptotics}, bringing it inside the limit, and splitting the resulting limit into three terms, \eqref{eq:int_for_asymptotics} is equivalent to
\begin{align}
\label{eq:split_integral1}
&\lim_{\ell \to \infty} \frac{1}{(2 \pi \bi)^k} \int_{\Gamma_1(n_\ell)^k} \hspace{-4ex}\tf(w_1,\ldots,w_k)\left(\frac{ \left(1+\frac{t^{\zeta+L_k-\eps(n_\ell)}(w_1+\cdots+w_k)}{n_\ell}\right)^{n_\ell} }{\prod_{i=1}^k(-tw_i;t)_{\eta_k(n_\ell)}} - \frac{e^{t^{\zeta+L_k-\eps(n_\ell)}(w_1+\cdots+w_k)}}{\prod_{i=1}^k(-tw_i;t)_{\infty}}  \right) \prod_{i=1}^k dw_i \\  
&+ \lim_{\ell \to \infty}\frac{1}{(2 \pi \bi)^k} \int_{\Gamma(r,n_\ell)^k \setminus \Gamma_1(n_\ell)^k} \frac{\tf(w_1,\ldots,w_k)\left(1+\frac{t^{\zeta+L_k-\eps(n_\ell)}(w_1+\cdots+w_k)}{n_\ell}\right)^{n_\ell} }{\prod_{i=1}^k(-tw_i;t)_{\eta_k(n_\ell)}} \prod_{i=1}^k dw_i \label{eq:split_integral2} \\
&-\lim_{\ell \to \infty} \frac{1}{(2 \pi \bi)^k} \int_{\tG^k \setminus \Gamma_1(n_\ell)^k} \frac{\tf(w_1,\ldots,w_k)e^{t^{\zeta+L_k-\eps(n_\ell)}(w_1+\cdots+w_k)}}{\prod_{i=1}^k(-tw_i;t)_{\infty}}\prod_{i=1}^k dw_i \label{eq:split_integral3}  = 0.
\end{align}
We will show each of the three lines \eqref{eq:split_integral1}, \eqref{eq:split_integral2} and \eqref{eq:split_integral3} above are $0$ separately. 

The fact that the third line \eqref{eq:split_integral3} is $0$ follows because (a) by \Cref{thm:series_to_contour_cL} the integral is finite when $\tG^k \setminus \Gamma_1(n_\ell)^k$ is replaced by the whole contour $\tG^k$, (b) $\Gamma_1(r,n-1) \subset \Gamma_1(r,n)$ for all (large enough) $n$, and (c) $\bigcup_n \Gamma_1(r,n)^k = \tG^k$ because $\xi(n) \to \infty$. 

For the other two lines \eqref{eq:split_integral1} and \eqref{eq:split_integral2}, we use the bounds on $\tf$ and $(z;t)_n$ which are quoted from \cite{van2023local}:

\begin{lemma}[{\cite[Lemma 4.9]{van2023local}}]\label{thm:tf_bound_quoted}
For any neighborhood $-1 \in U \subset \C$, there exist positive constants $C,c_2$ such that the bound
\begin{equation}
\label{eq:tf_bound_quoted}
|\tf(w_1,\ldots,w_k)| \leq C \prod_{i=1}^k e^{\frac{k-1}{2}(\log t^{-1}) \floor{\log_t |w_i|}^2 + c_2 \floor{\log_t |w_i|}}
\end{equation}
holds for any $w_1,\ldots,w_k \in \C \setminus (\D \cup U)$, where $\D$ is the open unit disc.
\end{lemma}

\begin{lemma}[{\cite[Lemma 4.8]{van2023local}}]\label{thm:qpoc_lower_bound_quoted}
For all $n \in \Z_{\geq 1} \cup \{\infty\}$, $\delta > 0$, and $z \in \C$ such that $|\Re(z) - (-t^{-i})| > \delta t^{-i}$, we have the bound
\begin{equation} \label{eq:qpoc_lower_bound_quoted}
|(z;t)_n| \geq \delta (t^{1/2};t)_\infty^2.
\end{equation}
\end{lemma}

\begin{lemma}[{\cite[Lemma 4.7]{van2023local}}]\label{thm:qp_bound_tG_quoted}
There exists a constant such that for any $n \in \Z_{\geq 1} \cup \{\infty\}$ and $w \in \tG$,
\begin{equation}
\abs*{\frac{1}{(-tw;t)_n}} \leq C.
\end{equation}
\end{lemma}

We now show that the limit \eqref{eq:split_integral1} is $0$. For $w_i \in \Gamma_1(n_\ell)$ we have $|t^{\eta_k(n_\ell)+1}w_i| \leq t^{\eta_k(n_\ell)-\xi(n_\ell)}$ (for $n$ large enough so $|-t^{- \xi(n_\ell)}\pm i| \leq t^{-\xi(n_\ell) - 1}$). Hence $(-t^{\eta_k(n_\ell)+1}w_i;t)_{\infty} = 1 + O(t^{\eta_k(n_\ell) - \xi(n_\ell)})$ and so
\begin{equation}
\label{eq:qp_inf_fin_bound}
\frac{1}{(-tw_i;t)_{\eta_k(n_\ell)}} = \frac{1}{(-tw_i;t)_\infty}(1+O(t^{\eta_k(n_\ell) - \xi(n_\ell)})).
\end{equation}

We have by writing $(1+x)^{n_\ell} = e^{n_\ell \log(1+x)}$ and Taylor expanding the logarithm that 
\begin{align}\label{eq:gamma1_prod_bound}
\begin{split}
\left(1+\frac{t^{\zeta+L_k-\eps(n_\ell)}\sum_{i=1}^k w_i}{n_\ell}\right)^{n_\ell}  &= e^{t^{\zeta+L_k-\eps(n_\ell)}\sum_{i=1}^k w_i + O(\sum_{i=1}^k w_i^2/n_\ell)} \\ 
&= e^{t^{\zeta+L_k-\eps(n_\ell)}\sum_{i=1}^k w_i(1+O(\eps(n_\ell)))}(1+O(\sum_{i=1}^k w_i^2/n_\ell))) \\ 
&= e^{t^{\zeta+L_k-\eps(n_\ell)}\sum_{i=1}^k w_i}(1+O(\eps(n_\ell)\sum_{i=1}^k w_i)+O(\sum_{i=1}^k w_i^2/n_\ell)).
\end{split}
\end{align}
By the definition of $\Gamma_1(n)$, for any $w_i \in \Gamma_1(n)$ we have
\begin{equation}
|w_i| \leq \text{const} \cdot t^{-\xi(n_\ell)},
\end{equation}
so the error terms in \eqref{eq:gamma1_prod_bound} are $O(t^{-\xi(n_\ell)}\eps(n_\ell))$ and $O(t^{-2\xi(n_\ell)}/n_\ell)$ respectively. Since $t^{\eta_k(n)} = O(1/n)$, the error term in \eqref{eq:qp_inf_fin_bound} is dominated by the latter term. Hence 
\begin{equation}\label{eq:error_convert}
\frac{\left(1+\frac{t^{\zeta+L_k-\eps(n_\ell)}\sum_{i=1}^k w_i}{n_\ell}\right)^{n_\ell} }{\prod_{i=1}^k(-tw_i;t)_{\eta_k(n_\ell)}}= \frac{e^{t^{\zeta+L_k}\sum_{i=1}^k w_i}}{\prod_{i=1}^k (-tw_i;t)_\infty}(1+O(t^{-\xi(n_\ell)}\eps(n_\ell))+O(t^{-2\xi(n_\ell)}/n_\ell))
\end{equation}
for all $w_i \in \Gamma_1(n_\ell)$. By using \Cref{thm:tf_bound_quoted} to bound $\tf$ and using \eqref{eq:error_convert} (together with \Cref{thm:qp_bound_tG_quoted}, which applies since $\Gamma_1(n_\ell) \subset \tG$, to bound the denominator $(-t w_i;t)_\infty$ by a constant) to bound the term inside parentheses, we have that the integrand in \eqref{eq:split_integral1} is 
\begin{equation}\label{eq:int_bound_gamma1}
O\left(e^{t^{\zeta+L_k}\sum_{i=1}^k \Re(w_i) + \text{const} \cdot (\log |w_i|)^2 }(O(t^{-\xi(n_\ell)}\eps(n_\ell))+O(t^{-2\xi(n_\ell)}/n_\ell))\right).
\end{equation}
Furthermore, the $w_i$-dependent exponential in \eqref{eq:int_bound_gamma1} is bounded above by an $\ell$-independent constant uniformly over all $(w_1,\ldots,w_k) \in \tG^k \supset \Gamma_1(n_\ell)^k$ (because the $\Re(w_i)$ cause it to shrink quickly), so we may absorb it into the $O$ constant in \eqref{eq:int_bound_gamma1}. Because each contour $\Gamma_1(n_\ell)$ has length $O(t^{-\xi(n_\ell)})$ and there are $k$ contours, multiplying the volume by the bound on the integrand yields the bound $O(t^{-(k+1)\xi(n_\ell)}\eps(n_\ell))+O(t^{-(k+2)\xi(n_\ell)}/n_\ell)$ for the integral \eqref{eq:split_integral1}. Because $\xi(n_\ell) \ll \log n_\ell$ and $\xi(n_\ell) \ll -\log \eps(n_\ell)$ by definition, this bound is $o(1)$, so we have shown the vanishing of the limit \eqref{eq:split_integral1}.

We now turn to the second line, \eqref{eq:split_integral2}. Because we have chosen $r$ so that $t^r < 1/k$, and $\Re(w_i) > -t^{-\eta_k(n_\ell) + r} $ everywhere on $\Gamma(n_\ell)$, we have
\begin{equation}\label{eq:sum_w_lower_bound}
\Re\left(\frac{t^{\zeta+L_k-\eps(n_\ell)}(w_1+\cdots+w_k)}{n_\ell}\right) > -1
\end{equation} 
for all $w_1,\ldots,w_k \in \Gamma(n_\ell)$; it is perhaps easiest to see this from the fact that the above expression is just $\Re(z_1+\cdots+z_k) > -kt^r > -1$ when expressed in terms of the $z_i$. Because $|x+iy| \leq |x| + |y|$, we have
\begin{align}\label{eq:abs_bound_line2}
\begin{split}
\abs*{1+\frac{t^{\zeta+L_k-\eps(n_\ell)}\sum_{i=1}^k w_i}{n_\ell}} &= \abs*{1+\Re\left(\frac{t^{\zeta+L_k-\eps(n_\ell)}\sum_{i=1}^k w_i}{n_\ell}\right)+\Im\left(\frac{t^{\zeta+L_k-\eps(n_\ell)}\sum_{i=1}^k w_i}{n_\ell}\right)\bi} \\ 
&\leq \abs*{1+\Re\left(\frac{t^{\zeta+L_k-\eps(n_\ell)}\sum_{i=1}^k w_i}{n_\ell}\right)} + \abs*{\Im\left(\frac{t^{\zeta+L_k-\eps(n_\ell)}\sum_{i=1}^k w_i}{n_\ell}\right)} \\
&\leq 1+\frac{t^{\zeta+L_k-\eps(n_\ell)}}{n_\ell}\sum_{i=1}^k (\Re(w_i) + 1),
\end{split}
\end{align}
where in the last inequality we have used \eqref{eq:sum_w_lower_bound} to remove the absolute value around the real part, and the fact that $|\Im(w_i)| \leq 1$ on our contours. By the elementary inequality
\begin{equation}
\left(1+\frac{x}{n}\right)^n \leq e^x \quad \quad \quad \quad \text{ for $x \geq -n$},
\end{equation}
it follows that 
\begin{equation}\label{eq:it's_like_exp_bound}
\abs*{1+\frac{t^{\zeta+L_k-\eps(n_\ell)}\sum_{i=1}^k w_i}{n_\ell}}^{n_\ell} \leq e^{t^{\zeta+L_k-\eps(n_\ell)}\sum_{i=1}^k w_i}.
\end{equation}
Finally, the $|(-tw_i;t)_{\eta_k(n_\ell)}|^{-1}$ terms in \eqref{eq:split_integral2} are bounded above by a constant, by applying \Cref{thm:qp_bound_tG_quoted} to the part of the contour contained in $\tG$, and applying \Cref{thm:qpoc_lower_bound_quoted} (one may take any $0 < \delta < 1/2$) to the vertical parts. Combining \eqref{eq:it's_like_exp_bound} with \Cref{thm:tf_bound_quoted} and this constant bound on $|(-tw_i;t)_{\eta_k(n_\ell)}|^{-1}$ yields the bound
\begin{multline}\label{eq:split2_integrand_bound}
\abs*{\frac{\tf(w_1,\ldots,w_k)\left(1+\frac{t^{\zeta+L_k-\eps(n_\ell)}(w_1+\cdots+w_k)}{n_\ell}\right)^{n_\ell} }{\prod_{i=1}^k(-tw_i;t)_{\eta_k(n_\ell)}}} \\ 
\leq \text{const} \cdot e^{t^{\zeta+L_k-\eps(n_\ell)}\sum_{i=1}^k \left((\Re(w_i)+1) +\frac{k-1}{2}(\log t^{-1}) \floor{\log_t |w_i|}^2 + c_2 \floor{\log_t |w_i|}\right)}
\end{multline}
for the integrand in \eqref{eq:split_integral2}. The right hand side factorizes, i.e.
\begin{align}\label{eq:bound_factorize}
\text{RHS\eqref{eq:split2_integrand_bound}} &= \prod_{i=1}^k b_\ell(w_i)  \\ 
b_\ell(w) &:= \text{const}^{1/k} e^{t^{\zeta+L_k-\eps(n_\ell)} \left((\Re(w)+1) +\frac{k-1}{2}(\log t^{-1}) \floor{\log_t |w|}^2 + c_2 \floor{\log_t |w|}\right)}.
\end{align}
Now note that
\begin{equation}\label{eq:gamma_union_split}
\Gamma(r,n_\ell)^k \setminus \Gamma_1(n_\ell)^k = \bigcup_{i=1}^k \Gamma(r,n_\ell)^{i-1} \times \Gamma_2(r,n_\ell) \times \Gamma(r,n_\ell)^{k-i}
\end{equation}
(not a disjoint union). Hence by symmetry of the integrand, the bound \eqref{eq:split2_integrand_bound}, and the factorization \eqref{eq:bound_factorize}, it suffices to show
\begin{equation}
\label{eq:sym_split_2}
\lim_{\ell \to \infty} \int_{\Gamma_2(r,n_\ell)} b_\ell(w) dw \cdot \left(\int_{\Gamma(r,n_\ell)} b_\ell(w) dw\right)^{k-1} = 0.
\end{equation}
This is an easy exercise: the integral over $\Gamma(r,n_\ell)$ is bounded above by a constant because the $e^{t^{\zeta+L_k-\eps(n_\ell)}w}$ term in $b_\ell$ decays very rapidly as $\Re(w) \to -\infty$, while the integral over $\Gamma_2(r,n_\ell)$ goes to $0$ for the same reason since $\sup_{w \in \Gamma_2(r,n_\ell)} \Re(w) \to -\infty$ as $\ell \to \infty$. This shows \eqref{eq:sym_split_2}, which shows that the second line \eqref{eq:split_integral2} is $0$ and hence completes the proof.
\end{proof}


\begin{proof}[Proof of {\Cref{thm:jordan_limit_intro}}]
We follow the proof of \cite[Theorem 1.2]{van2023local} with appropriate substitutions. Suppose for the sake of contradiction that \eqref{eq:haar_metric_cvg} does not hold. Then there exists some $\eps > 0$, $k \in \Z_{\geq 1}$ and some subsequence $(n_j)_{j \geq 1}$ of $\Z_{\geq 1}$ such that 
\begin{equation}\label{eq:distances_too_large}
D_\infty\left((\rank(A_n^{i-1}) - \rank(A_n^i) - \floor{\log_q n})_{1 \leq i \leq k}, (\cL^{(i)}_{q^{-1},q^{\{\log_q N_j\}}})_{1 \leq i \leq k}\right) > \eps
\end{equation}
for all $j$. Since the fractional parts $\{\log_q n_j\}$ always lie in the compact set $[0,1]$, there is some $\zeta\in [-1,0]$ and further subsequence $(\tN_j)_{j \geq 1}$ of $(n_j)_{j \geq 1}$ such that 
\begin{equation}\label{eq:frac_zeta_cvg}
\lim_{j \to \infty} \{\log_q \tN_j\} = -\zeta,
\end{equation}
and in particular $-\log_q \tN_j$ converges to $\zeta$ in $\R/\Z$. Hence by \Cref{thm:jordan_limit_nometric}, for all $k \geq 1$ we have
\begin{equation}\label{eq:apply_matrix_bulk}
(\rank(A_{\tN_j}^{i-1}) - \rank(A_{\tN_j}^i) - [\log_q \tN_j+\zeta])_{1 \leq i \leq k}  \to (\cL^{(i)}_{q^{-1},q^{-\zeta}})_{1 \leq i \leq k}
\end{equation}
in distribution as $j \to \infty$. By \eqref{eq:frac_zeta_cvg}, $[\log _{q}\tN_j+\zeta] = \floor{\log_q \tN_j}$ for all $j$ sufficiently large, hence \eqref{eq:apply_matrix_bulk} implies that for all $k \geq 1$,
\begin{equation}\label{eq:apply_matrix_bulk2}
(\rank(A_{\tN_j}^{i-1}) - \rank(A_{\tN_j}^i) - \floor{\log_q \tN_j})_{1 \leq i \leq k} \to (\cL^{(i)}_{q^{-1},q^{-\zeta}})_{1 \leq i \leq k}
\end{equation}
in distribution as $j \to \infty$. Equivalently,
\begin{equation}\label{eq:D_to_zero_zeta}
\lim_{j \to \infty} D_\infty\left((\rank(A_{\tN_j}^{i-1}) - \rank(A_{\tN_j}^i) - \floor{\log_q \tN_j})_{1 \leq i \leq k},(\cL^{(i)}_{q^{-1},q^{-\zeta}})_{1 \leq i \leq k}\right) = 0.
\end{equation}
The integral representation in \Cref{thm:series_to_contour_cL} and the integrand bound \cite[Lemma 4.4]{van2023local} together imply that for each $\vec{L} \in \Sig_k$, the probability
\begin{equation}
\Pr((\cL^{(i)}_{q^{-1},q^{-\zeta}})_{1 \leq i \leq k} = \vec{L})
\end{equation}
depends continuously on $\zeta$. Hence by \eqref{eq:frac_zeta_cvg}, 
\begin{equation}\label{eq:probs_unif_cont}
\lim_{j \to \infty} D_\infty\left((\cL^{(i)}_{q^{-1},q^{\{\log_q \tN_j\}}})_{1 \leq i \leq k},(\cL^{(i)}_{q^{-1},q^{-\zeta}})_{1 \leq i \leq k}\right) = 0
\end{equation}
(this requires uniform continuity of the probabilities over all $\vec{L}$, but this follows from the stated continuity of each individual probability since the sum of probabilities is $1$). The triangle inequality for $D_\infty$ and the equations \eqref{eq:D_to_zero_zeta}, \eqref{eq:probs_unif_cont} thus imply
\begin{align}
\begin{split}
\lim_{j \to \infty} D_\infty\left((\rank(A_{\tN_j}^{i-1}) - \rank(A_{\tN_j}^i) - \floor{\log_q \tN_j})_{1 \leq i \leq k}, (\cL^{(i)}_{q^{-1},q^{\{\log_q \tN_j\}}})_{1 \leq i \leq k}\right) = 0,
\end{split}
\end{align}
but this contradicts our assumption \eqref{eq:distances_too_large}. Hence this assumption is false, i.e. the conclusion \eqref{eq:haar_metric_cvg} of \Cref{thm:jordan_limit_intro} holds, and this completes the proof.
\end{proof}


\end{document}